\documentclass[final,reqno]{amsart}

\usepackage{graphicx}
\usepackage{amsmath,amsthm,amsfonts,amscd,amssymb}
\usepackage[color,notref,notcite]{showkeys}
\usepackage{url}
\usepackage{subeqnarray}
\usepackage[color,notref,notcite]{showkeys} 
\definecolor{labelkey}{rgb}{0.6,0,1}
\usepackage{enumitem}
\usepackage{algorithm}
\usepackage{algpseudocode}
\usepackage{citesort}

\theoremstyle{plain}
\newtheorem{theorem}{Theorem}[section]

\newtheorem{lemma}[theorem]{Lemma}

\newtheorem{corollary}[theorem]{Corollary}

\theoremstyle{definition}
\newtheorem{definition}[theorem]{Definition}
\newtheorem{example}[theorem]{Example}

\def\bhyp#1{\begin{equation}\label{#1}\begin{array}{c}}
\def\ehyp{\end{array}\end{equation}}

\newcounter{cst}

\theoremstyle{remark}
\newtheorem{remark}[theorem]{Remark}

\numberwithin{equation}{section}
\numberwithin{figure}{section}

\newcommand{\RR}{{\mathbb R}}
\newcommand{\NN}{{\mathbb N}}

\def\O{\Omega}
\def\dsp{\displaystyle}
\def\bfn{\mathbf{n}}

\def\disc{{\mathcal D}}
\def\mesh{{\mathcal M}}
\newcommand{\polyd}{{\mathcal T}}
\def\edges{{\mathcal E}}
\def\edge{\sigma}
\def\xcv{x_K}
\def\cv{K}

\newcommand{\edgescv}{{{\edges}_K}}  
\newcommand{\edgesext}{{{\edges}_{\rm ext}}} 
\newcommand{\edgesint}{{{\edges}_{\rm int}}} 
\newcommand{\centers}{\mathcal{P}}

\def\dr{\partial}
\newcommand{\centeredge}{\overline{x}_\edge} 

\newcommand{\cI}{{\mathcal I}}

\newcommand{\x}{\pmb{x}}

\newif\ifcorr\corrtrue
\usepackage[normalem]{ulem}
\definecolor{violet}{rgb}{0.580,0.,0.827}

\def\bpsi{{\boldsymbol \psi}}

\newcommand{\ud}{\, \mathrm{d}} 
\def\div{\mathop{\rm div}}

\title[Analysis of schemes for anisotropic reaction diffusion models]{A gradient Discretisation Method For Anisotropic Reaction Diffusion Models with applications to the dynamics of brain tumours}

\author{Yahya Alnashri}
\address[Yahya Alnashri]{Department of Mathematics, Al-Qunfudah University College, Umm Al-Qura University, Saudi Arabia}
\email{yanashri@uqu.edu.sa}

\author{Hasan Alzubaidi}
\address[Hasan Alzubaidi]{Department of Mathematics, Al-Qunfudah University College, Umm Al-Qura University, Saudi Arabia}
\email{hmzubaidi@uqu.edu.sa}



\date{\today}

\begin{document}
\newcommand{\subscript}[2]{$#1 _ #2$}

\begin{abstract}
A gradient discretisation method (GDM) is an abstract setting that designs the unified convergence analysis of several numerical methods for partial differential equations and their corresponding models. In this paper, we study the GDM for anisotropic reaction diffusion problems, based on a general reaction term, with Neumann and Dirichlet boundary conditions. With natural regularity assumptions on the exact solution, the framework enables us to provide proof of the existence of weak solutions for the problem, and to obtain a uniform--in--time convergence for the discrete solution and a strong convergence for its discrete gradient. It also allows us to apply non conforming numerical schemes to the model on a generic grid; (the Crouzeix--Raviart scheme and the hybrid mixed mimetic (HMM) methods). Numerical experiments using the HMM method are performed to study the growth of glioma tumours in heterogeneous brain environment. The dynamics of their highly diffusive nature is also measured using the fraction anisotropic measure. The validity of the HMM is examined further using four different mesh types. The results indicate that the dynamics of the brain tumour is still captured by the HMM scheme, even in the event of a highly heterogeneous anisotropic case performed on the mesh with extreme distortions.
\end{abstract}

\maketitle

\section{Introduction}
\label{introduction}
In this paper, we study the following anisotropic reaction diffusion model:
\begin{subequations}\label{problem-rm}
\begin{align}
\partial_t \bar c(\x,t)-\div\big({\bf A}(\x)\nabla \bar c(\x,t) \big)&=F(\bar c(\x,t)), \quad (\x,t) \in \O\times (0,T),\label{rm-strong1}\\
\nabla\bar c(\x,t) \cdot \mathbf{n}&= 0, \quad  (\x,t) \in \partial\O\times (0,T), \label{rm-strong2}\\
\bar c(\cdot,0)&=c_{\rm ini} \mbox{ on } \O. \label{rm-strong3}
\end{align}
\end{subequations}
With a particular choice of the reaction term $F$, the model can expresses brain tumours, in which unknown $\bar c$ represents the density of cancerous cells, and homogeneous zero flux boundary conditions state that there is no diffusion of tumor cells out-with the brain region and the domain $\O \subset \RR^d, (d\geq 1)$.


\par Invasive diffuse brain tumors that frequently recur in spite of the improvements in therapy plans in recent years, gliomas can often prove fatal within six months to a year of recurrence \cite{1,2}. Treatment plans generally encompass chemotherapy, radiotherapy, and surgery, but despite all the advances in these therapies, it is still rare for such tumors to be completely cured \cite{3,4}. A primary problem in the treatment of glioma is that they are highly diffusive and have heterogeneous invasion rates that create invisible antigen tumors that cannot be detected with current resolutions of imaging \cite{3,4,5}. The heterogeneous spread pattern is probably attributable to the anisotropic invasion of glioma cells through the brain's aligned structures, e.g., the bundled neural fiber tracts that characterize white matter \cite{5,6,7}. The brain chiefly comprises two forms of tissue, white matter and gray matter. Gray matter comprises glial and neuronal cell bodies controlling the activity of the brain; white matter is used by glioma cells as an invasion route between areas of gray matter \cite{5}. Evidence from research has suggested that within white matter tumor diffusion is anisotropic, while within gray matter it is isotropic \cite{7,8}.
\par One means of forecasting the invasive pathways is mathematical modeling, and researchers have developed several macroscopic models using diffusion tensor imaging (DTI) data for informing the architecture of white matter and simulating a glioma's non-uniform growth \cite{9,10}. DTI is a means of imaging that takes measures of the way in which water molecules are an-isotropically diffused within a tissue and could potentially offer predictions of how a tumor will expand and so direct therapy planning \cite{11,12}.
\par Reaction diffusion models have been created in order to forge a link between medical imaging and tumor growth models. Tracqui et al. \cite{13} suggested one of the earliest reaction diffusion models that integrated medical imaging information, specifically tumor size and the brain's geometry. Cruywagen at al \cite{14} expanded on this concept and suggested employing two tumor cell populations. Burgess et al.\cite{15} extended the model into three dimensions, emphasizing the part played in glioma growth by diffusion. Many researchers have focused on the macroscopic process of glioma expansion, employing reaction diffusion type partial differential equations (PDEs), with proliferation represented by the reaction term and infiltration represented by the diffusion term \cite{5}.

\par The first models made an assumption of homogeneous and isotropic growth with ${\bf A}$ set as a constant and scalar diffusion; subsequent reaction models have been employed for the description of the growth of such tumors within the brain's heterogeneous environment. One example was Swanson et al. \cite{16} who examined a spatially heterogeneous diffusion coefficient, with ${\bf A}$ set as being notably greater in white matter compared to gray matter for the description of the swifter invasion noted in these regions. Swanson et al. \cite{17} expanded to a three--dimensional model, using data regarding gray and white matter areas harvested from anatomical imaging. 
Furthermore, studies of medical images have demonstrated that tumor cells have a tendency to follow water diffusion patterns; measurement of this can be achieved by employing magnetic resonance diffusion tensor imaging (MR-DTI). Expanding upon tumor cells' differential motility within a variety of tissues in such a context, numerous researchers have built inhomogenous diffusion tensors ${\bf A}$ in order to model tumor cell diffusion using images from the diffusion tensor (DTI) \cite{9,10,18}. 
DTI offers data regarding anisotropic diffusion as shown by eigenvalues (magnitude) and eigenvectors (direction). The tensors were constructed by employing anisotropic (ellipsoid) diffusion for white matter, and isotropic (spherical) tensors for gray matter \cite{11,12,18}.

\par{}
While there has been much theoretical study of this model as part of a general theory of reaction diffusion equations \cite{45}, it is typically very rare to solve models that capture the gross behaviour of glioma tumors in heterogeneous brain tissue based on data imaging. A number of different numerical approaches for the description of glioma tumors' heterogeneous rate of invasion and the dynamics of their highly diffusive nature (mostly without full convergence analysis) have been employed, for example, see \cite{23,24,25,26,22,27,28,29,new-51}. However, few studies have so far been paid to the convergence analysis of non conforming methods for the reaction diffusion equation and its corresponding models. For example, finite volume method is proposed to approximate the convection diffusion reaction equations \cite{new-49} and the Nagumo--type equations \cite{new-50}. 

\par In this work, we develop a gradient discretisation method (GDM) to the model \eqref{problem-rm}. The (GDM) is a generic framework that can find a unified convergence analysis, in which be applicable to multiple families numerical schemes instead of conducting individual convergence analysis for each numerical scheme. The efficiency of the framework is ensured under a limited number of properties (depending on the model). See the monograph \cite{30} for details. By approximating the model \eqref{problem-rm} by the GDM, we afford
\begin{itemize}
\item \emph{Comprehensive convergence analysis of numerical schemes for anisotropic reaction diffusion models.}

\item \emph{Analysis that carried out based on standard regularity assumptions on the data, which can be established in realistic models.}
\item \emph{Analysis that can be easily extended to the model \eqref{problem-rm} with different boundary conditions.}

\item \emph{Implementation of finite volume methods which can be performed on generic grids to deal with highly heterogeneous anisotropic problems.}
\end{itemize}

\par This paper is organised as follows. Section \ref{sec-ds} introduces the discrete elements to approximate the considered reaction diffusion model together with some basic properties to guarantee the convergence of approximation schemes. This is followed by Section \ref{sec-mr}, which is concerned with the approximate scheme and its convergence results. Section \ref{sec-ex} presents two examples of non conforming schemes that fit into the gradient discretisation method and have been not yet proposed to any type of anisotropic reaction diffusion models; the Crouzeix--Raviart scheme and the hybrid mixed mimetic (HMM) methods. They simultaneously provide an approximation of the solution and its gradient on a generic grid. In Section \ref{sec-proof} we use the compactness argument to prove our convergence results (Theorem \ref{th-mr}) under natural assumptions on data. This approach relies on establishing energy estimates on a discrete solution under standard assumptions on the model data. One benefit of our analysis is to prove the existence of a weak solution to \eqref{problem-rm}, which does not need to be assumed. In Section \ref{sec-fn} we show that the GDM can successfully be extended to the reaction diffusion equations subject to non homogeneous Dirichlet boundary condition. Finally, in Section \ref{sec-nr}, some numerical experiments using the hybrid mixed mimetic (HMM) method are provided to study the growth of brain tumours in heterogeneous environment. The validity of the HMM scheme is examined further using anisotropic cases performed on different generic meshes.       

\section{Discrete setting}\label{sec-ds}
The idea of the gradient discretisation method is to construct appropriate gradient discretisations, made of discrete space and operators, to approximate the continuous model provided that written in the weak sense. Writing the equivalent weak formulation  with replacing the continuous elements by the discrete ones yields a numerical scheme called a gradient scheme (GS). Let us now start with recalling the notions of the gradient discretisation method, that are suitable to discretise partial differential equations with homogenous Neumann boundary conditions defined as in \cite{30}.  

\begin{definition}[gradient discretisations]\label{def-gd-rm}
Let $\O$ be an open subset of $\RR^d$ (with $d \geq 1$) and $T>0$. A gradient discretisations for the anisotropic reaction diffusion model \eqref{problem-rm} is $\disc=(X_{\disc},\Pi_\disc, \nabla_\disc, J_\disc, (t^{(n)})_{n=0,...,N}) )$, where
\begin{itemize}
\item the set of discrete unknowns $X_{\disc}$ is a finite dimensional vector space on $\RR$,
\item the linear mapping $\Pi_\disc : X_{\disc} \to L^2(\O)$ is the reconstructed function, 
\item the linear mapping $\nabla_\disc : X_{\disc} \to L^2(\O)^d$ is a reconstructed gradient, which must be chosen such that
\begin{equation}\label{norm-disc}
\| \varphi ||_\disc = || \Pi_\disc \varphi ||_{L^2(\O)}+ || \nabla_\disc \varphi ||_{L^2(\O)^d} 
\end{equation}
is a norm on $X_{\disc}$, 
\item the linear continuous mapping $J_\disc: L^2(\O) \to X_\disc$ is an interpolation operator for the
 initial conditions,
\item $t^{(0)}=0<t^{(1)}<....<t^{(N)}=T$ are time steps. 
\end{itemize}
\end{definition}

For $(\varphi^{(n)})_{n\in\NN} \subset X_\disc$, let us define the piecewise constant in time functions $\Pi_\disc \varphi: (0,T]\to L^2(\O)$, $\nabla_\disc \varphi: (0,T]\to L^2(\O)^d$,and $\delta_\disc \varphi : (0,T] \to L^2(\O)$ as follows. For a.e $\x\in\O$, for all $n\in\{0,...,N-1 \}$ and for all $t\in (t^{(n)},t^{(n+1)}]$,
\begin{equation*}
\begin{aligned}
&\Pi_\disc \varphi(\x,0):=\Pi_\disc \varphi^{(0)}(\x), \quad \Pi_\disc \varphi(\x,t):=\Pi_\disc \varphi^{(n+1)}(\x),\\
&\nabla_\disc \varphi(\x,t):=\nabla_\disc \varphi^{(n+1)}(\x), \quad
\delta_\disc \varphi(t)=\delta_\disc^{(n+\frac{1}{2})}\varphi:=\frac{\Pi_\disc(\varphi^{(n+1)}-\varphi^{(n)})}{\delta t^{(n+\frac{1}{2})}},
\end{aligned}
\end{equation*}
with setting $\delta t^{(n+\frac{1}{2})}=t^{(n+1)}-t^{(n)}$ and $\delta t_\disc=\max_{n=0,...,N-1}\delta t^{(n+\frac{1}{2})}$.

Due to the flexibility of choices of gradient discretisations, various numerical scheme families fit into the GDM, starting from conforming, non conforming and mixed finite elements to nodal mimetic finite differences and hybrid mimetic mixed methods \cite{31,32,33,34,35,36}. Two examples are presented below and taken from \cite{30}.

\begin{example}
In conforming $\mathbb P 1$ finite element method, the discrete space $X_{\disc}$ consists of vectors of values at the nodes of the mesh, the reconstructed operator $\Pi_\disc \varphi$ is the piecewise linear continuous function that takes these values at the nodes, and $\nabla_\disc \varphi = \nabla(\Pi_\disc \varphi)$.

\par For non conforming $\mathbb P 1$ finite element method, the discrete space $X_\disc$ is made of piecewise linear functions on a triangle $\mathcal T$, which are continuous at the edge mid-points. Unlike the conforming methods, the reconstructed operators are respectively $\Pi_\disc=\boldmath{Id}$ and $\nabla_\disc=\nabla_B$ (the broken gradient), i.e.,
 \[
 \mbox{for all } \varphi \in X_\disc, \mbox{ for all a triangle } \tau \in \mathcal T,\; \nabla_B \varphi(\x)=\nabla(\varphi_{|\tau}), \; \forall \x \in \tau.
 \]
\end{example}

As a means of constructing converging schemes, the gradient discretisations elements must take the properties of continuous space and operators. The quality of the choice of gradient discretisations elements can be measured through two parameters, correspond to errors in an interpolation of function by smooth ones and a discrete Stokes formula. 

\begin{definition}[Consistency]\label{def:cons-rm}
For $\varphi\in H^1(\O)$, define
$S_{\mathcal{D}} : H^1(\O)\to [0, +\infty)$ by
\begin{equation}\label{cons-rm1}
S_{\mathcal{D}}(\varphi)= 
\min_{w\in X_\disc}\left(\| \Pi_{\mathcal{D}} w - \varphi \|_{L^{2}(\Omega)}
+ \| \nabla_{\mathcal{D}} w - \nabla \varphi \|_{L^{2}(\Omega)^{d}}\right).
\end{equation}
A sequence $(\mathcal{D}_{m})_{m \in \mathbb{N}}$ of gradient discretisations is \emph{consistent} if, as $m \to \infty$ 
\begin{itemize}
\item for all $\varphi \in H^1(\O)$, $S_{\disc_m}(\varphi) \to 0$,
\item for all $\varphi \in L^2(\O)$, $\Pi_{\disc_m}J_{\disc_m}\varphi \to \varphi$ strongly in $L^2(\O)$,
\item $\delta t_{\disc_m} \to 0$.
\end{itemize}
\end{definition}
\begin{definition}[Limit--conformity]\label{def:lconf-rm}
For $\bpsi \in H_{\rm div}$, define $W_{\mathcal{D}} : H_{\rm div} \to [0, +\infty)$ by
\begin{equation}\label{long-rm}
W_{\mathcal{D}}(\bpsi)
 = \sup_{w\in X_\disc\setminus \{0\}}\frac{\Big|\dsp\int_{\Omega}(\nabla_{\mathcal{D}}w\cdot \bpsi + \Pi_{\mathcal{D}}w \div (\bpsi)) \ud \x \Big|}{|| w||_\disc },
\end{equation}
where $H_{\rm div}=\{\bpsi \in L^2(\O)^d\;:\; {\rm div}\bpsi \in L^2(\O),\; \bpsi\cdot\bfn=0 \mbox{ on } \partial\O  \}$.
 
A sequence $(\disc_m)_{m\in \NN}$ of gradient discretisations is \emph{limit-conforming} if for all $\bpsi \in H_{\rm div}$, $W_{\disc_m}(\bpsi) \to 0$, as $m \to \infty$.
\end{definition}
Lastly, to address non linearity coming from the reaction term, the operator $\Pi_\disc$ is required to meet the compactness properties, as outlined below.
\begin{definition}[Compactness]\label{def:compact}
A sequence of gradient discretisations $(\disc_m)_{m\in\NN}$ in the sense of Definition \ref{def-gd-rm} is \emph{compact} if for any sequence $(\varphi_m )_{m\in\NN} \in X_{\disc_m}$, such that $(|| \varphi_m ||_{\disc_m})_{m\in \NN}$ is bounded, the sequence $(\Pi_{\disc_m}\varphi_m )_{m\in \NN}$ is relatively compact in $L^2(\O)$.
\end{definition}

\section{Main Results}\label{sec-mr}
As explained before, the GDM relies on the weak formulation. We introduce here the weak formulation and its approximation scheme created from the gradient discretisations presented in the previous section. We consider the problem \eqref{problem-rm} under the following assumptions: 

\begin{subequations}\label{assump-rm}
\begin{gather}
\O \mbox{ is an open bounded connected subset of $\RR^d\; (d \geq 1)$ with a Lipschitz boundary},\nonumber\\
T >0, c_{\rm ini} \in L^2(\O),
\end{gather}
\begin{gather}
{\bf A}: \O \to \mathbb M_d(\RR) \mbox{ is a measurable function (where $\mathbb M_d(\RR)$ is the set of $d \times d$ matrices)}\nonumber\\
\mbox{and there exists $\underline \lambda $, $\overline \lambda >0$ such that for a.e. $\x \in \O$},\\
\mbox{${\bf A}(\x)$ is symmetric with eigenvalues in $[\underline \lambda, \overline \lambda]$, }\nonumber
\end{gather}
\begin{gather}\label{assump-rm-3}
\mbox{$F: \RR \to \RR$ is a polynomial function, such that } F \in L^2(\O\times(0,T)),\nonumber\\
\mbox{and there exists constants $C_{F1},\; C_{F2} \geq 0$ such that for any $s\in \RR$,}\nonumber\\
|F(s)| \leq C_{F1} + C_{F2}|s|.
\end{gather}
\end{subequations}

Under assumptions \eqref{assump-rm}, $\bar c$ is said to be a weak solution of \eqref{problem-rm} if
\begin{equation}\label{rm-weak}
\left.
\begin{aligned}
&\bar c \in L^2(0,T;H^1(\O)) \mbox{ and, for all } \bar\varphi \in L^2(0,T;H^1(\O)),\\
&\mbox{such that, $\partial_t \bar\varphi \in L^2(\O\times(0,T))$ and $\bar\varphi(\cdot,T)=0$,}\\
&-\dsp\int_0^T \dsp\int_\O \bar c(\x,t) \partial_t \bar\varphi(\x,t) \ud \x  \ud t
-\dsp\int_\O c_{\rm ini}(\x)\bar\varphi(\x,0) \ud \x\\ 
{}&\quad+\dsp\int_0^T\int_\O {\bf A}(\x)\nabla \bar{c}(\x,t) \cdot \nabla \bar\varphi(\x,t)\ud \x \ud t\\
{}&\quad\quad\quad\quad\quad= \dsp\int_0^T\dsp\int_\O F(\bar c(\x,t))\bar\varphi(\x,t) \ud \x \ud t.
\end{aligned}
\right.
\end{equation}

Note that this problem can be written in the sense of distributions. Take $\bar\varphi\in C^\infty(\O\times(0,T))$. Integration by parts and the density of $C^\infty([0,T];H^1(\O))$ in $L^2(0,T;H^1(\O))$, the following holds:
\begin{equation*}
\left.
\begin{aligned}
&\bar c \in L^2(0,T;H^1(\O)) \cap C([0,T];L^2(\O));\;
\partial_t\bar c \in L^2(0,T;H^{-1}(\O)),\\
&\bar c(\cdot,0)=c_{\rm ini} \mbox{ and, for all } \bar\varphi \in L^2(0,T;H^1(\O))\\
&\dsp\int_0^T \langle \partial_t\bar c(\x,t),\bar\varphi(\x,t) \rangle_{H^{-1},H^1}\ud t
+\dsp\int_0^T\int_\O {\bf A}(\x)\nabla \bar{c}(\x,t) \cdot \nabla \bar\varphi(\x,t)\ud \x \ud t\\
{}&\quad\quad\quad\quad\quad= \dsp\int_0^T\dsp\int_\O F(\bar c(\x,t))\bar\varphi(\x,t) \ud \x \ud t.
\end{aligned}
\right.
\end{equation*}


The gradient scheme for the problem \eqref{rm-weak}, that coming from the gradient discretisations $\disc$ in the sense of Definition \eqref{def-gd-rm} is given by
\begin{equation}\label{rm-disc-pblm}
\begin{aligned}
&\mbox{find a family $(c^{(n)})_{n=0,...,N} \in X_\disc^{N+1}$, $c^{(0)}=J_\disc c_{\rm ini}$, }\\
&\mbox{ and for all $n=0,...,N-1$, $c^{(n+1)}$ satisfies}\\
&\dsp\int_\O \delta_\disc^{(n+\frac{1}{2})} c(\x) \Pi_\disc \varphi(\x)
+ \dsp\int_\O {\bf A}(\x) \nabla_\disc c^{(n+1)}(\x) \cdot \nabla_\disc \varphi(\x)\ud \x\\
&\qquad=\dsp\int_\O F(\Pi_\disc c^{(n+1)}(\x))\Pi_\disc \varphi(\x) \ud \x, \quad \forall \varphi \in X_\disc.
\end{aligned}
\end{equation}

The following theorem states the main theoretical results of this work.
\begin{theorem}\label{th-mr}
Let assumptions \eqref{assump-rm} hold and $(\disc_m)_{m\in\NN}$ be a sequence of gradient discretisations, that is consistent, limit-conforming and compact. For $m \in \NN$, let $c_m$ be a solution to the gradient scheme \eqref{rm-disc-pblm} with $\disc=\disc_m$. Then there exists a solution $\bar c$ of \eqref{rm-weak} and a subsequence of gradient discretisations, denoted again by $(\disc_m)_{m\in\NN}$, such that, as $m \to \infty$,
\begin{subequations}
\label{eq-thm-conv}
\begin{align}
\dsp\sup_{t\in[0,T]}\| \Pi_{\disc_m}c_m(t) - \bar c(t) \|_{L^2(\O)} \to 0,\\
|| \nabla_{\disc_m}c_m - \nabla\bar c ||_{L^2(\O \times (0,T))^d} \to 0.
\end{align}
\end{subequations}
\end{theorem}

\begin{remark}
Note that the proof of this theorem does not require to initially assume the existence and the uniqueness of the continuous weak solution. The convergence analysis provided here can establish the existence of at least one solution $\bar c$ to the problem \eqref{rm-weak}, see Corollary \ref{cor-rdm}.
\end{remark}

\section{$2$--Examples covered by the analysis}\label{sec-ex}
The analysis designed in this work can be applicable to many numerical schemes. We show here that two different non conforming methods can be expressed as the gradient schemes formats \eqref{rm-disc-pblm}. To do so, let us begin by recalling the definition of a generic polyhedral mesh as in  \cite{31}.

\begin{definition}[Polytopal mesh]\label{def:polymesh}~
Let $\Omega$ be a bounded polytopal open subset of $\RR^d$ ($d\ge 1$). 
A polytopal mesh of $\O$ is given by $\polyd = (\mesh,\edges,\centers)$, where:
\begin{enumerate}
\item $\mesh$ is a finite family of non empty connected polytopal open disjoint subsets of $\O$ (the cells) such that $\overline{\O}= \dsp{\cup_{K \in \mesh} \overline{K}}$.
For any $K\in\mesh$, $|K|>0$ is the measure of $K$ and $h_K$ denotes the diameter of $K$.

\item $\edges$ is a finite family of disjoint subsets of $\overline{\O}$ (the edges of the mesh in 2D,
the faces in 3D), such that any $\edge\in\edges$ is a non empty open subset of a hyperplane of $\RR^d$ and $\edge\subset \overline{\O}$.
We assume that for all $K \in \mesh$ there exists  a subset $\edgescv$ of $\edges$
such that $\dr K  = \dsp{\cup_{\edge \in \edgescv}} \overline{\edge}$. 
We then set $\mesh_\edge = \{K\in\mesh\,:\,\edge\in\edgescv\}$
and assume that, for all $\edge\in\edges$, $\mesh_\edge$ has exactly one element
and $\edge\subset\partial\O$, or $\mesh_\edge$ has two elements and
$\edge\subset\O$. 
$\edgesint$ is the set of all interior faces, i.e. $\edge\in\edges$ such that $\edge\subset \O$, and $\edgesext$ the set of boundary
faces, i.e. $\edge\in\edges$ such that $\edge\subset \dr\O$.
For $\edge\in\edges$, the $(d-1)$-dimensional measure of $\edge$ is $|\edge|$,
the centre of mass of $\edge$ is $\centeredge$, and the diameter of $\edge$ is $h_\edge$.

\item $\centers = (x_K)_{K \in \mesh}$ is a family of points of $\O$ indexed by $\mesh$ and such that, for all  $K\in\mesh$,  $\xcv\in K$ ($\xcv$ is sometimes called the ``centre'' of $\cv$). 
We then assume that all cells $K\in\mesh$ are  strictly $\xcv$-star-shaped, meaning that 
if $x\in \overline{K}$ then the line segment $[\xcv,x)$ is included in $K$.
\end{enumerate}
For a given $K\in \mesh$, let $\bfn_{K,\sigma}$ be the unit vector normal to $\sigma$ outward to $K$
and denote by $d_{K,\sigma}$ the orthogonal distance between $x_K$ and $\sigma\in\mathcal E_K$.
The size of the discretisation is $h_\mesh=\sup\{h_K\,:\; K\in \mesh\}$.
\end{definition}

\subsection{The Crouzeix--Raviart scheme}\label{sr-scheme}
It is recently known as the non conforming $\mathbb P 1$ finite element and linked to Stokes models \cite{Crouzeix}. Set as in \cite{new-52}
\begin{enumerate}
\item The space of unknowns is $X_\disc=\{ \varphi=(\varphi_\edge)_{\edge\in\edges}\;:\; v_\edge \in \RR   \}$.
\item The linear reconstructed function $\Pi_\disc$ is defined based on affine non conforming finite element basis function $e_K^\edge$, and give by
\[
\begin{aligned}
&\forall \varphi\in X_\disc, \forall K\in \mesh, \mbox{ for a.e. } \x \in K,\\
&\Pi_\disc \varphi(\x) =\dsp\sum_{\edge\in\edges_K}\varphi_\edge e_K^\edge(\x),
\end{aligned}
\]
\item The reconstructed gradient is defined by
\[
\begin{aligned}
&\forall \varphi\in X_\disc, \forall K\in \mesh,\\
&(\nabla_\disc \varphi)_{|K}=\nabla[(\Pi_\disc \varphi)_{|K}]=\dsp\sum_{\edge\in\edges_K}\varphi_\edge \nabla e_K^\edge.
\end{aligned}
\]
It is called a broken gradient (i.e. a piecewise constant on the cells).
\item The interpolant $J_\disc: L^2(\O) \to X_\disc$ is defined by:
\[
\forall w \in L^2(\O)\;:\; J_\disc w=(w_\edge)_{\edge\in \edges},
\mbox{ where $w_\edge=\dsp\frac{1}{|\edge|}\dsp\int_\edge w(\x) \ud \x$}.
\]
\end{enumerate}

The Crouzeix--Raviart scheme of Problem \eqref{rm-weak} is the gradient scheme \eqref{rm-disc-pblm} with the gradient discretisations constructed above. It is proved in \cite{new-52} that this gradient discretisations satisfies the three properties; the consistency, the limit conformity, and the compactness. Therefore, Theorem \ref{th-mr} provides the convergence of the Crouzeix--Raviart scheme for the anisotropic reaction diffusion model.

\subsection{The hybrid mixed mimetic (HMM) method}\label{sec:HMM}
It is found by \cite{42} that the HMM method is a framework gathering three different schemes: the (mixed-hybrid) mimetic finite differences methods, the hybrid finite volume method, and the mixed finite volume methods. The method can also be compatible with a generic mesh with non orthogonality assumptions. Let $\polyd$ be a polytopal mesh of $\O$ defined in Definition \ref{def:polymesh}.
\begin{enumerate}
\item The discrete space is 
\[
X_{\disc}=\{ v=((\varphi_{K})_{K\in \mathcal{M}}, (\varphi_{\sigma})_{\sigma \in \mathcal{E}})\;:\; \varphi_{K},\, \varphi_{\sigma} \in \RR,
\}.
\]
\item The non conforming a piecewise affine reconstruction $\Pi_\disc$ is defined by
\[
\begin{aligned}
&\forall \varphi\in X_\disc, \forall K\in \mesh, \mbox{ for a.e. } \x \in K,\\
&\Pi_\disc \varphi=\varphi_K\mbox{ on $K$}.
\end{aligned}
\]
\item The reconstructed gradients is piecewise constant on the cells (broken gradient), defined by
\[
\begin{aligned}
&\forall \varphi\in X_\disc,\; \forall K\in\mathcal M,\,\forall \sigma\in\mathcal E_K,\\
&\nabla_\disc \varphi=\nabla_{K}\varphi+
\frac{\sqrt{d}}{d_{K,\sigma}}R_K(\varphi)\mathbf{n}_{K,\sigma} \mbox{ on } D_{K,\edge},
\end{aligned}
\]
where a cell--wise constant gradient $\nabla_K(\varphi)$ and a stabilisation term $R_K(\varphi)$ are respectively defined by:
\[
\nabla_{K}\varphi= \dsp\frac{1}{|K|}\sum_{\sigma\in \edgescv}|\sigma|\varphi_\edge\mathbf{n}_{K,\sigma} \mbox{ and } R_K(\varphi)=(\varphi_\edge - \varphi_K - \nabla_K \varphi\cdot(\centeredge-x_K))_{\edge\in\edges_K}.
\]

\item The interpolant $J_\disc: L^2(\O) \to X_\disc$ is defined by:
\[
\begin{aligned}
&\forall w \in L^2(\O)\;:\; J_\disc w=((w_K)_{K\in \mesh},(w_\edge)_{\edge\in \edges}),\\
&\forall K\in\mesh,\; w_K=\dsp\frac{1}{|K|}\dsp\int_K w(\x) \ud \x \mbox{ and } \forall \edge\in\edges,\; w_\edge=0.
\end{aligned}
\]
\end{enumerate}

The HMM scheme of Problem \eqref{rm-weak} is the gradient scheme \eqref{rm-disc-pblm} with the gradient discretisations constructed above, it reads
\begin{equation}\label{hmm-rdm}
\begin{aligned}
&\mbox{find $(c^{(n)})_{n=0,...,N} \in X_\disc^{N+1}$, such that $c^{(0)}=J_{\disc}c_{\rm ini}$},\\ 
&\mbox{and for all $n=0,...,N-1$, $c^{(n+1)}$ satisfies, for all $\varphi \in X_\disc$} \\
&\dsp\sum_{K\in \mesh}\frac{|K|}{\delta t^{(n+\frac{1}{2})}}\Big( c_K^{(n+1)}-c_K^{(n)} \Big)\varphi_K
+\dsp\sum_{K\in \mesh}|K|{\bf A}(\x)\nabla_K c^{(n+1)}\cdot \nabla_K \varphi\\
&\quad+\dsp\sum_{K\in \mesh}(R_K\varphi)^T \mathbb B_K R_K (c^{(n+1)})
=\dsp\sum_{K\in \mesh}\varphi_K\dsp\int_K F(c_K^{(n+1)}) \ud \x,
\end{aligned}
\end{equation}
where $\mathbb B_K$ is a symmetric positive definite matrix of size $\mbox{Card}(\edges_K)$.

It is proved in \cite[Chapter 9]{30} that this gradient discretisations satisfies the three properties; the consistency, the limit conformity and the compactness. Therefore, Theorem \ref{th-mr} provides the convergence of the HMM scheme for the anisotropic reaction diffusion model.

The HMM scheme \eqref{hmm-rdm} can be presented in classical finite volume formats. With considering the linear fluxes $c\mapsto \mathcal F_{K,\sigma}(c)$ (for $K\in\mesh$ and $\sigma\in\edges_K$) defined:
for all $K \in \mesh$ and all $c,w \in X_\disc$,
\begin{align*}
\sum_{\sigma \in \mathcal{E}_K}|\sigma| \mathcal F_{K,\sigma}(c)
(w_K-w_\sigma)={}&\dsp\int_K {\bf A}(\x)\nabla_\disc c\cdot\nabla_\disc w \ud \x,
\end{align*}

Then Problem \eqref{hmm-rdm} can be written as, for all $n=0,...,N-1$,
\begin{align*}
\frac{|K|}{\delta t^{(n+\frac{1}{2})}} \Big( c^{(n+1)}-c^{(n)} \Big)&+\sum_{\sigma \in \mathcal{E}_K}|\sigma|\mathcal F_{K,\sigma}(u^{(n+1)})=\dsp\int_K F(c_K^{(n+1)}) \ud \x, \quad \forall K \in \mesh\\
\mathcal F_{K,\sigma}(c^{(n+1)})+\mathcal F_{L,\sigma}(c^{(n+1)})&= 0, \quad \forall \sigma\in\mathcal E_{\rm int}\mbox{ with }
\mesh_\sigma=\{K,L\},\\
\mathcal F_{K,\sigma}(c^{(n+1)})&=0, \quad \forall K \in \mesh\,,\forall \sigma \in \mathcal{E}_K \mbox{ such that }\sigma\subset  \partial\O.
\end{align*}


\section{Proof of The Main Results}\label{sec-proof}
In order to prove our convergence results, the discrete solution and its gradient must possess some energy estimates, which are established in the following lemmas.
\begin{lemma}[Estimates]
\label{lemma-est-rm}
Under assumptions \eqref{assump-rm}, let $\disc$ be a gradient discretisations and let $c \in X_\disc$ be a solution to the gradient scheme \eqref{rm-disc-pblm}. Then there exists a constant $C_1\geq 0$ depending only on $C_{F1}$, $C_{F2}$, $\O$ and $|| \Pi_\disc c^{(0)} ||_{L^2(\O)}$, such that
\begin{equation}\label{est-rm}
\dsp\sup_{t\in[0,T]}||\Pi_\disc c(t) ||_{L^2(\O)}
+
||\nabla_\disc c ||_{ L^2(\O \times (0,T))^d }
\leq C_1 .
\end{equation}
\end{lemma}

\begin{proof}
Take $n \in \{ 0,...,N-1\}$ and $\varphi:=\delta t^{ (n+\frac{1}{2}) }c^{(n+1)}$ in Scheme \eqref{rm-disc-pblm}, to get
\begin{equation*}
\begin{aligned}
\dsp\int_\O \Big(\Pi_\disc c^{(n+1)}(\x)&-\Pi_\disc c^{(n)}(\x)\Big) \Pi_\disc c^{(n+1)}(\x) \ud \x
+\dsp\dsp\int_{t^{(n)}}^{t^{(n+1)}}\int_\O {\bf A}(\x)|\nabla_\disc c^{(n+1)}(\x)|^2 \ud \x \ud t\\
&\quad= \dsp\int_{t^{(n)}}^{t^{(n+1)}}\int_\O F(\Pi_\disc  c^{(n+1)}(\x)) \Pi_\disc c^{(n+1)}(\x)  \ud \x \ud t.
\end{aligned}
\end{equation*}
For $\alpha, \beta \in \RR$, $(\alpha-\beta)\alpha \geq \frac{1}{2}( |\alpha|^2 -|\beta|^2 )$. Applying this inequality to the first term in the above equality yields
\begin{equation*}
\begin{aligned}
\frac{1}{2}\dsp\int_\O \Big[ |\Pi_\disc c^{(n+1)}(\x)|^2&-|\Pi_\disc c^{(n)}(\x)|^2 \Big] \ud \x
+\underline \lambda\dsp\int_{t^{(n)}}^{t^{(n+1)}}\int_\O |\nabla_\disc c^{(n+1)}(\x)|^2 \ud \x \ud t\\
&\quad\leq \dsp\int_{t^{(n)}}^{t^{(n+1)}}\int_\O F(\Pi_\disc  c^{(n+1)}(\x)) \Pi_\disc c^{(n+1)}(\x)  \ud \x \ud t.
\end{aligned}
\end{equation*}
Sum on $n=0,...,m-1$, for some $m=0,...,N$:
\begin{equation}\label{eq-est10}
\begin{aligned}
\frac{1}{2}\dsp\int_\O \Big[ |\Pi_\disc c^{(m)}(\x)|^2&-|\Pi_\disc c^{(0)}(\x)|^2 \Big] \ud \x
+\underline \lambda\dsp\int_{0}^{t^{(m)}}\int_\O |\nabla_\disc c(\x)|^2 \ud \x \ud t \\
&\quad\leq \dsp\int_{0}^{t^{(m)}}\int_\O F(\Pi_\disc  c(\x)) \Pi_\disc c(\x)  \ud \x \ud t.
\end{aligned}
\end{equation}
Apply the Cauchy--Schwarz inequality to the right--hand side, to obtain,  
\begin{equation*}
\begin{aligned}
\frac{1}{2}\dsp\int_\O \Big[ |\Pi_\disc c^{(m)}(\x)|^2&-|\Pi_\disc c^{(0)}(\x)|^2 \Big] \ud \x
+\underline \lambda\dsp\int_{0}^{t^{(m)}}\int_\O |\nabla_\disc c(\x)|^2 \ud \x \ud t\\
&\quad\leq || F||_{L^2(\O\times (0,T))} \; || \Pi_\disc c||_{L^2(\O\times (0,T))}.
\end{aligned}
\end{equation*}
Due to assumptions \eqref{assump-rm-3}, one has
\begin{equation*}
\begin{aligned}
\frac{1}{2}\dsp\int_\O \Big[ |\Pi_\disc c^{(m)}(\x)|^2&-|\Pi_\disc c^{(0)}(\x)|^2 \Big] \ud \x
+\underline\lambda\dsp\int_{0}^{t^{(m)}}\int_\O |\nabla_\disc c(\x,t)|^2 \ud \x \ud t\\ 
&\leq C_1 || \Pi_\disc c^{(m)} ||_{L^2(\O\times (0,T))} + C_2 || \Pi_\disc c^{(m)} ||_{L^2(\O\times (0,T))}^2.
\end{aligned}
\end{equation*}
Then using the Young's inequality, with $\varepsilon$ satisfying $\frac{1}{2}-\frac{\varepsilon}{2}+C_{F2} >0$, to the right--hand side of the above inequality, we have
\begin{equation*}
\begin{aligned}
\frac{1}{2}\dsp\int_\O \Big[ |\Pi_\disc c^{(m)}(\x)|^2&-|\Pi_\disc c^{(0)}(\x)|^2 \Big] \ud \x
+\underline\lambda\dsp\dsp\int_{0}^{t^{(m)}}\int_\O |\nabla_\disc c(\x,t)|^2 \ud \x \ud t\\
&\quad\leq \frac{C_{F1}^2}{2\varepsilon}+(\frac{\varepsilon}{2}+C_{F2}) || \Pi_\disc c^{(m)} ||_{L^\infty(0,T;L^2(\O))}^2.
\end{aligned}
\end{equation*}
Take the supremum on $m=0,...,N$ to conclude Estimate \eqref{est-rm}.
\end{proof}

\begin{corollary}\label{cor-rdm}
Let $\disc$ be a gradient discretisations. Under assumptions \eqref{assump-rm}, the gradient scheme \eqref{rm-disc-pblm} admits at least one solution $c\in X_\disc$.
\end{corollary}

\begin{proof}
At each time step $n+1$, \eqref{rm-disc-pblm} describes square non linear equations on $c^{(n+1)}$. For a given $w \in X_\disc$, $c\in X_\disc$ is a solution to the linear square system
\begin{equation}\label{rm-disc-pblm-lin}
\begin{aligned}
&\dsp\int_\O \Pi_{\disc}\dsp\frac{c^{(n+1)}-c^{(n)}}{\delta t^{(n+\frac{1}{2})}}(\x) \Pi_\disc \varphi(\x)
+ \dsp\int_\O {\bf A}(\x)\nabla_\disc c^{(n+1)}(\x) \cdot \nabla_\disc \varphi(\x)\ud \x\\
&\qquad=\dsp\int_\O F(\Pi_\disc w) \Pi_\disc \varphi(\x) \ud \x, \quad \forall \varphi \in X_\disc.
\end{aligned}
\end{equation}

Using arguments similar to the proof of Lemma \ref{lemma-est-rm}, we can obtain
\begin{equation*}
\begin{aligned}
&|| \Pi_\disc c^{(n+1)} ||_{L^2(\O)} + || \nabla_\disc c^{(n+1)} ||_{L^2(\O)^d} \leq
C|| F ||_{L^2(\O)} + || \Pi_{\disc} c^{(n)}||_{L^2(\O)}.
\end{aligned}
\end{equation*}
where $C$ not depending on $c^{(n+1)}$. This shows that the kernel of the matrix built from the linear system only has the zero vector. The matrix is therefore invertible. We then can define the mapping $T:X_\disc \to X_\disc$ by $T(w)=c$ with $c$ is the solution to \eqref{rm-disc-pblm-lin}. Since $T$ is continuous, Brouwer's fixed point establishes the existence of a solution $c^{(n+1)}$ to the system at time step $n + 1$. 
\end{proof}
To reach the standard compactness, we need to establish a bound on the discrete time derivative. The space $\Pi_\disc(X_\disc) \subset L^2(\O)$ is therefore equipped with a dual norm defined as follows.
\begin{definition}[Dual norm on $\Pi_\disc(X_\disc)$]\label{def-dual}
Let $\disc$ be a gradient discretisations. The dual norm $|| \cdot ||_{\star,\disc}$ on $\Pi_\disc(X_\disc)$ is defined by
\begin{equation}\label{eq-dual}
\forall w\in \Pi_\disc(X_\disc),\;
|| w ||_{\star,\disc}=\dsp\sup\Big\{ \dsp\int_\O w(\x)\Pi_\disc \varphi(\x)\ud \x \; :\; \varphi\in X_\disc, || \varphi ||_\disc =1
\Big\}. 
\end{equation}
\end{definition}

\begin{lemma}[Estimate on the dual norm of $\delta_\disc c$]
\label{lemma-est-dual}
Assume \eqref{assump-rm} holds and let $\disc$ be a gradient discretisations. If $c \in X_\disc$ is a solution to the gradient scheme \eqref{rm-disc-pblm}, then there exists a constant $C_2$ depending only on $C_1$, $C_{F1}$, $C_{F2}$ and $|| \Pi_\disc c^{(0)} ||_{L^2(\O)}$, such that
\begin{equation}\label{eq-est-dual}  
\dsp\int_0^T || \delta_\disc c(t) ||_{\star,\disc}^2 \ud t \leq C_2,
\end{equation}
where the dual norm $|| \cdot ||_{\star,\disc}$ is defined by \eqref{eq-dual}.
\end{lemma}

\begin{proof}
Take $\varphi =w$ in \eqref{rm-disc-pblm}. Use the Cauchy--Schwarz inequality to get, thanks to assumptions \eqref{assump-rm}
\[
\begin{aligned}
\dsp\int_\O \delta_\disc^{(n+\frac{1}{2})} &c(\x) \Pi_\disc w(\x) \ud \x\\
&\quad\leq \overline \lambda || \nabla_\disc c^{(n+1)} ||_{L^2(\O\times(0,T))^d} || \nabla_\disc w ||_{L^2(\O\times(0,T))^d}\\
&\qquad+ \Big( C_{F1}+C_{F2} || \Pi_\disc c^{(n+1)} ||_{L^2(\O\times(0,T))} \Big) || \Pi_\disc w||_{L^2(\O)}\\
&\quad\leq || w ||_\disc \Big( \overline \lambda || \nabla_\disc c^{(n+1)} ||_{L^2(\O\times(0,T))^d} + C_{F1}+C_{F2} || \Pi_\disc c^{(n+1)} ||_{L^2(\O\times(0,T))} \Big).
\end{aligned}
\]
The conclusion then follows from taking the supremum over $w\in X_\disc$ with $|| w ||_\disc=1$, multiplying by $\delta t^{(n+1)}$, summing over $n=0,...,N-1$ and Estimate \eqref{est-rm}.

\end{proof}

\subsection*{Proof of Theorem \ref{th-mr}}
The proof relies on the compactness arguments as in \cite{30}, and is divided into four stages.\\

\noindent {\bf Step 1: Compactness results.}
Given that Estimate \eqref{est-rm} and the two properties of the gradient discretisations (consistency and the limit--conformity), \cite[Lemma 4.8]{30} shows that there exists $\bar c \in L^2(0,T;H^1(\O))$, such that, up to a subsequence, $\Pi_{\disc_m}c_m \to \bar c$ weakly in $L^2(0,T,L^2(\O))$, and $\nabla_{\disc_m}c_m \to \nabla\bar c$ weakly in $L^2(0,T,L^2(\O)^d)$. Estimate \eqref{eq-est-dual} together with the consistency, limit--conformity and compactness, \cite[Theorem 4.14]{30} proves that, in fact,  the convergence of $\Pi_{\disc_m}c_m$ to $\bar c$ is strong in $L^2(\O\times(0,T))$.
\vskip 1pc
\noindent {\bf Step 2: Convergence of the scheme.} We show that $\bar c$ mentioned in the first step is solution to the continuous problem. Take $\bar\psi \in L^2(0,T;L^2(\O))$ such that $\partial_t \bar \psi \in L^2(\O\times(0,T))$ and $\bar\psi(T,\cdot)=0$. The interpolation results in \cite[Lemma 4.10]{30} provides $w_m=(w_m^{(n)})_{n=0,...,N_m} \in X_{\disc_m}^{N_m+1}$, such that $\Pi_{\disc_m}w_m \to \bar\psi$ in $L^2(0,T;L^2(\O))$ and $\delta_{\disc_m}w_m \to \partial_t \bar\psi$ strongly in $L^2(\O\times(0,T))$. Set $\psi=\delta t_m ^{(n+\frac{1}{2})}w_m^{(n)}$ as a test function in Scheme \eqref{rm-disc-pblm} and sum on $n=0,...,N_m-1$ to get
\begin{equation}\label{eq-part1}
\begin{aligned}
\dsp\sum_{n=0}^{N_m-1}\dsp\int_\O [ \Pi_{\disc_m}c_m^{(n+1)}(\x)&-\Pi_{\disc_m}c^{(n)}(\x) ]\Pi_{\disc_m}w_m^{(n)}(\x) \ud x\\
&\quad+\dsp\int_0^T\dsp\int_\O {\bf A}(\x)\nabla_{\disc_m}c_m(\x,t) \cdot \nabla_{\disc_m}w_m(\x,t) \ud \x \ud t\\
&=\dsp\int_0^T\dsp\int_\O F(\Pi_{\disc_m}c(\x,t))\Pi_{\disc_m}w_m(\x,t) \ud x \ud t.
\end{aligned}
\end{equation}
Apply the discrete integration by part formula \cite[Eq. (D.15)]{30} to the right hand side in the above equality and use the fact $w^{(N)}=0$ to obtain 
\begin{equation*}
\begin{aligned}
-\dsp\int_0^T \dsp\int_\O \Pi_{\disc_m}c_m(\x,t)&\delta_{\disc_m}w_m(\x,t) \ud \x \ud t 
-\int_\O \Pi_{\disc_m}c_m^{(0)}(\x)\Pi_{\disc_m}w_m^{(0)}(\x) \ud \x\\
&+\dsp\int_0^T\dsp\int_\O {\bf A}(\x)\nabla_{\disc_m}c_m(\x,t) \cdot \nabla_{\disc_m}w_m(\x,t) \ud \x \ud t\\
&\quad=\dsp\int_0^T\dsp\int_\O F(\Pi_{\disc_m}c(\x,t))\Pi_{\disc_m}w_m(\x,t) \ud x \ud t.
\end{aligned}
\end{equation*}
Thanks to the strong convergence of $\Pi_{\disc_m}c_m$ and \eqref{assump-rm-3}, the dominated convergence theorem implies that $F(\Pi_{\disc_m}c_m) \to F(\bar c)$ in $L^2(\O\times(0,T))$. By the consistency, $\Pi_{\disc_m}c_m^{(0)}=\Pi_{\disc_m}J_{\disc_m}c_{\rm ini} \to c_{\rm ini}$ in $L^2(\O)$. Now, we can pass to the limit $m\to \infty$ in each of the terms above to see that $\bar c$ is a solution to \eqref{rm-weak}, thanks again to the weak and strong convergence established previously.
\vskip 1pc
\noindent {\bf Step 3: Convergence of $\Pi_{\disc_m}c_m$ in $L^\infty(0,T;L^2(\O))$.} For $s\in [0,T]$, let $(s_m)_{m\in\NN} \subset [0,T]$ be a sequence such that $s_m \to s$, as $m \to \infty$. Let $k(m) \in \{ 0,..., N_m-1\}$ such that $s_m 
\in (t^{(k(m))}, t^{(k(m)+1)}]$. As followed in the proof of Lemma \ref{lemma-est-rm}, we obtain as the discrete estimate \eqref{eq-est10} with $\disc_m$ and $k_m$,
\begin{equation}\label{eq-prf-thm1}
\begin{aligned}
\frac{1}{2}\dsp\int_\O &(\Pi_{\disc_m} c(\x,s_m))^2 \ud \x\\
&\leq \frac{1}{2}\dsp\int_\O (\Pi_{\disc_m} J_{\disc_m}c_{\rm ini}(\x))^2 \ud x 
- \dsp\int_0^{t^{(k(m))}}\dsp\int_\O {\bf A}(\x)(\nabla_{\disc_m}c(\x,t))^2 \ud \x \ud t\\
&\leq \dsp\int_0^{t^{(k(m))}}\dsp\int_\O F(\Pi_{\disc_m}c(\x,t)) \Pi_{\disc_m}c(\x,t) \ud \x \ud t.
\end{aligned}
\end{equation}
Let $\chi_I$ be the characteristic function of $I$. Now, as $m \to \infty$, we have
\begin{equation*}
\begin{aligned}
&\Pi_{\disc_m}c_m \to \bar c \mbox{ strongly in } L^2(\O \times (0,T)), \mbox{ and,}\\
&\chi_{[0,t^{(k(m))}]}\nabla\bar c \to \chi_{[0,s]} \nabla\bar c \mbox{ strongly in } L^2(\O \times (0,T))^d.
\end{aligned}
\end{equation*}
It is obvious to write
\begin{equation*}
\begin{aligned}
\dsp\int_0^s\dsp\int_\O &{\bf A}(\x)(\nabla\bar c(\x,t))^2 \ud \x \ud t\\
&=\dsp\int_0^{t^{(k(m))}}\dsp\int_\O \chi_{[0,s]}{\bf A}(\x)(\nabla\bar c(\x,t))^2 \ud \x \ud t\\
&=\dsp\lim_{m\to \infty}\dsp\int_0^T\dsp\int_\O \chi_{[0,t^{(k(m))}]}{\bf A}(\x)\nabla\bar c(\x,t) \cdot \nabla_{\disc_m}c_m(\x,t) \ud \x \ud t\\
&\leq \dsp\liminf_{m\to \infty}\Big( || \chi_{[0,t^{(k(m))}]}\nabla\bar c ||_{L^2(\O\times(0,T))^d} \cdot || \chi_{[0,t^{(k(m))}]}{\bf A}(\x)\nabla_{\disc_m}c_m ||_{L^2(\O\times(0,T))^d} \Big)\\
&= || \chi_{[0,s]}\nabla\bar c ||_{L^2(\O\times(0,T))^d} \cdot \dsp\liminf_{m\to \infty} || \chi_{[0,t^{(k(m))}]}{\bf A}(\x)\nabla_{\disc_m}c_m ||_{L^2(\O\times(0,T))^d}. 
\end{aligned}
\end{equation*}
Dividing by $\| \chi_{[0,s]}\nabla\bar c \|_{L^2(\O\times(0,T))^d}$ leads to
\begin{equation}\label{eq-liminf1}
\begin{aligned}
\dsp\int_0^s\dsp\int_\O {\bf A}(\x)(\nabla\bar c(\x,t))^2 &\ud \x \ud t \\&\leq \dsp\liminf_{m\to\infty}\dsp\int_0^{t^{(k(m))}}\dsp\int_\O {\bf A}(\x)( \nabla_{\disc_m}u_m(\x,t))^2 \ud \x \ud t.
\end{aligned}
\end{equation}

Combined with passing to limit superior in \eqref{eq-prf-thm1}, this gives
\begin{equation}\label{eq-20}
\begin{aligned}
\dsp\limsup_{m\to \infty}\frac{1}{2}\dsp\int_\O (\Pi_{\disc_m}&c_m(\x,s_m))^2 \ud \x\\
&\leq \frac{1}{2}\dsp\int_\O c_{\rm ini}(\x)^2 \ud \x
-\dsp\int_0^s\dsp\int_\O {\bf A}(\x) (\nabla\bar u(\x,t))^2 \ud \x \ud t\\
&\quad+\dsp\int_0^s\dsp\int_\O F(\bar c(\x,t)) \bar c(\x,t)
 \ud \x \ud t.
\end{aligned}
\end{equation}

Plugging $\varphi=\bar c\chi_{[0,s]}(t)$ in Problem \eqref{rm-weak} and integrating by part, one has
\begin{equation}\label{eq-multline-1}
\begin{aligned}
\frac{1}{2}\dsp\int_\O (\bar c(\x,s))^2 \ud \x &+ \dsp\int_0^s\dsp\int_\O {\bf A}(\x)(\nabla \bar c(\x,t))^2 \ud \x \ud t\\
&\quad=\frac{1}{2}\dsp\int_\O c_{\rm ini}(\x)^2 \ud \x
+\dsp\int_0^s\dsp\int_\O F(\bar c(\x,t)) \bar c(\x,t) \ud \x \ud t.
\end{aligned}
\end{equation}
Putting it all together (\eqref{eq-20}, \eqref{eq-multline-1}), we have
\begin{equation}\label{eq-limsup-3}
\dsp\limsup_{m\to \infty}\dsp\int_\O (\Pi_{\disc_m}c_m(\x,s_m))^2 \ud \x
\leq \dsp\int_\O \bar c(\x,s)^2 \ud \x.
\end{equation}

\cite[Theorem 4.19]{30} and Estimates \eqref{est-rm} and \eqref{eq-est-dual} states that $(\Pi_{\disc_m}c_m)_{m\in\NN}$ converges to $\bar c$ weakly in $L^2(\O)$ uniformly in $[0,T]$. Hence, $\Pi_{\disc_m}c_m(\cdot,s_m)$ converges to $\bar c(\cdot,s)$ weakly in $L^2(\O)$, as $m \to \infty$. Estimate \eqref{eq-limsup-3} with basic justifications in Hilbert, this convergence strongly holds in $L^2(\O)$. Since $\bar c:[0,T] \to L^2(\O)$ is continuous, apply \cite[Lemma C.13]{30} to conclude the proof. 

\vskip 1pc
\noindent {\bf Step 4: Strong Convergence of $\nabla_{\disc_m}c_m$.} Since the non linearity does not act on gradients, the proof can be obtained as in \cite{YJ2014} without additional assumptions. It can be written
\begin{equation}\label{es-eq}
\begin{aligned}
\dsp\int_0^T\dsp\int_\O (\nabla_{\disc_m}c_m(\x,t)&- \nabla\bar c(\x,t))\cdot(\nabla_{\disc_m}c_m(\x,t)- \nabla\bar c(\x,t)) \ud \x \ud t\\
&=\dsp\int_0^T\dsp\int_\O \nabla_{\disc_m}c_m(\x,t) \cdot \nabla_{\disc_m}c_m(\x,t) \ud \x \ud t\\
&\quad-\dsp\int_0^T\dsp\int_\O \nabla_{\disc_m}c_m(\x,t) \cdot \nabla\bar c(\x,t) \ud \x \ud t\\
&\quad-\dsp\int_0^T\dsp\int_\O \nabla\bar c(\x,t) \cdot (\nabla_{\disc_m}c_m(\x,t)-\nabla\bar c(\x,t)) \ud \x \ud t.
\end{aligned}
\end{equation}
Take $\bar\varphi:=c_m$ in Scheme \eqref{rm-disc-pblm}. We can pass to the limit superior, to obtain, thanks to choosing $\varphi = \bar c$ in \eqref{rm-weak} and ${\bf A}(\x)={\bf Id}$ in the continuous and discrete problems
\begin{equation*}
\begin{aligned}
\dsp\limsup_{m\to \infty}\dsp\int_0^T\dsp\int_\O \nabla_{\disc_m}&c_m(\x,t)\cdot \nabla_{\disc_m}c_m(\x,t) \ud x \ud t\\
&=\dsp\int_0^T\dsp\int_\O F(\bar c)\bar c(\x,t) \ud \x \ud t
-\dsp\int_0^T\dsp\int_\O \partial_t \bar c(\x,t) \bar c(\x,t) \ud \x \ud t
\\
&=\dsp\int_0^T\dsp\int_\O \nabla\bar c(\x,t)\cdot \nabla\bar c(\x,t)\ud \x \ud t.
\end{aligned}
\end{equation*}
This relation and the weak convergence of $\nabla_{\disc_m} c_m$ enable to pass to the limit in \eqref{es-eq} to complete the proof.

\section{The Case Of non homogenous Dirichlet boundary condition}\label{sec-fn}

\subsection{Continuous Setting} 
Anisotropic reaction diffusion equations with non homogenous Dirichlet boundary conditions have various applications appearing in nerve conduction, biophysics, ecology and clinical medicine \cite{new-50}, for example. We consider here the following model:
\begin{subequations}\label{fn-problem-rm}
\begin{align}
\partial_t \bar c-\div\big({\bf A}(\x,t)\nabla \bar c \big)&=F(\bar c) \mbox{ in } \O\times (0,T),\label{fn-rm-strong1}\\
\bar c(\x,t)&=g(\x,t)   \mbox{ on } \partial\O\times (0,T), \label{fn-rm-strong2}\\
\bar c(\cdot,0)&=c_{\rm ini} \mbox{ on } \O, \label{fn-rm-strong3}
\end{align}
\end{subequations}
where $\O$, $F$ and $c_{\rm ini}$ are as in Section \ref{sec-mr}. The diffusion tensor ${\bf A}(\x,t)$ is positive definite and $g$ is a trace of a function in $L^2(0,T;H^1(\O))$ whose time derivative is in $L^2(0,T;H^{-1}(\O))$. 

Take $w \in L^2(0,T;H^1(\O))$ such that the trace of $w$ is the function $g$. The weak solution of \eqref{fn-problem-rm} is seeking $\hat c=\bar c-w$ satisfying
\begin{equation}\label{fn-weak}
\left.
\begin{aligned}
&\widehat c \in L^2(0,T;H_0^1(\O)) \cap C([0,T];L^2(\O));\;
\partial_t \widehat c \in L^2(0,T;H^{-1}(\O)),\;
\bar c(\cdot,0)=c_{\rm ini}\\ 
&\dsp\int_0^T \langle \partial_t\widehat c(\x,t),\bar\varphi(\x,t) \rangle_{H^{-1},H^1}\ud t
+\dsp\int_0^T\int_\O {\bf A}(\x,t)\nabla \widehat{c}(\x,t) \cdot \nabla \bar\varphi(\x,t)\ud \x \ud t\\
{}&\quad\quad\quad\quad\quad= \dsp\int_0^T\dsp\int_\O F(\widehat c(\x,t))\bar\varphi(\x,t) \ud \x \ud t,\quad \forall \bar\varphi \in L^2(0,T;H_0^1(\O)).
\end{aligned}
\right.
\end{equation}

\subsection{Discrete Setting}
We give here the approximate scheme of the considered problem and its convergence results together with examples of particular numerical schemes.
      
\begin{definition}\label{fn-def-gd-rm}
Let $\O$ be an open subset of $\RR^d$ (with $d \geq 1$) and $T>0$. A gradient discretisations for time--dependent problem including a non homogenous Dirichlet boundary conditions is $\disc=(X_{\disc},\mathcal I_{\disc,\dr}, \Pi_\disc, \nabla_\disc, J_\disc, (t^{(n)})_{n=0,...,N}) )$, where
\begin{itemize}
\item the set of discrete unknowns $X_{\disc}=X_{\disc,0}\oplus X_{\disc,\dr\O}$ is the direct sum of two finite dimensional spaces on $\RR$, corresponding respectively to the interior unknowns and to the boundary unknowns,
\item the linear mapping $\mathcal I_{\disc,\dr}: H^{\frac{1}{2}}(\dr\O) \to X_{\disc,\dr\O}$ is an interpolation operator for the trace,
\item the reconstructed function $\Pi_\disc$, reconstructed gradient $\nabla_\disc$, the interpolation operator $J_\disc$ and the discrete time steps $(t^{(n)})_{n=0,...,N}$ are as in Definition \ref{def-gd-rm}, such that the discrete gradient must be chosen so that $\| \cdot ||_\disc := || \nabla_\disc \cdot ||_{L^2(\O)^d}$ defines a norm on $X_{\disc,0}$.
\end{itemize}
\end{definition}


The accuracy of this gradient discretisations is measured through the three properties, consistency, limit--conformity and compactness, defined below.
     
\begin{definition}[Consistency]\label{fn-def:cons-rm}
For $\varphi\in H^1(\O)$, define
$S_{\mathcal{D}} : H^1(\O)\to [0, +\infty)$ by
\begin{equation}
\begin{aligned}
S_{\mathcal{D}}(\varphi)= 
\min\Big\{\| \Pi_{\mathcal{D}} w - \varphi \|_{L^{2}(\Omega)}
+ \| \nabla_{\mathcal{D}} w - \nabla \varphi \|_{L^{2}(\Omega)^{d}}\;:\;
w\in X_\disc,\\ 
\mbox{ such that } w-\mathcal I_{\disc,\dr}\gamma\varphi \in X_{\disc,0}\Big\},
\end{aligned}
\end{equation}
where $\gamma$ is the trace operator acting on functions in $L^2(0,T;H^1(\O))$.

A sequence $(\mathcal{D}_{m})_{m \in \mathbb{N}}$ of gradient discretisations in the sense of Definition \ref{fn-def-gd-rm} is \emph{consistent} if, as $m \to \infty$ 
\begin{itemize}
\item for all $\varphi \in H^1(\O)$, $S_{\disc_m}(\varphi) \to 0$,
\item for all $\varphi \in L^2(\O)$, $\Pi_{\disc_m}J_{\disc_m}\varphi \to \varphi$ strongly in $L^2(\O)$,
\item $\delta t_{\disc_m} \to 0$.
\end{itemize}
\end{definition}

\begin{definition}[Limit--conformity]\label{fn-def:lconf-rm}
For $\bpsi \in \overline H_{\rm div}$, define $W_{\mathcal{D}} : \overline H_{\rm div} \to [0, +\infty)$ by
\begin{equation}\label{long-rm}
W_{\mathcal{D}}(\bpsi)
 = \sup_{w\in X_{\disc,0}\setminus \{0\}}\frac{\Big|\dsp\int_{\Omega}(\nabla_{\mathcal{D}}w\cdot \bpsi + \Pi_{\mathcal{D}}w \div (\bpsi)) \ud \x \Big|}{|| w||_\disc }.
\end{equation}
where $\overline H_{\rm div}=\{\bpsi \in L^2(\O)^d\;:\; {\rm div}\bpsi \in L^2(\O)\}$.
 
A sequence $(\disc_m)_{m\in \NN}$ of gradient discretisations in the sense of Definition \ref{fn-def-gd-rm} is \emph{limit-conforming} if for all $\bpsi \in \overline H_{\rm div}$, $W_{\disc_m}(\bpsi) \to 0$, as $m \to \infty$.
\end{definition}

\begin{definition}[Compactness]\label{fn-def:compact}
A sequence of gradient discretisations $(\disc_m)_{m\in\NN}$ in the sense of Definition \ref{fn-def-gd-rm} is \emph{compact} if for any sequence $(\varphi_m )_{m\in\NN} \in X_{\disc_m}$, such that $(|| \varphi_m ||_{\disc_m})_{m\in \NN}$ is bounded, the sequence $(\Pi_{\disc_m}\varphi_m )_{m\in \NN}$ is relatively compact in $L^2(\O)$.
\end{definition}


If $\disc$ is a gradient discretisations in the sense of Definition \eqref{fn-def-gd-rm}, the corresponding gradient scheme is given by
\begin{equation}\label{fn-disc-pblm}
\begin{aligned}
&\mbox{find a family $(c^{(n)})_{n=0,...,N} \in \mathcal I_{\disc,\dr}g+ X_{\disc,0}^{N+1}$, $c^{(0)}=J_\disc c_{\rm ini}$, }\\
&\mbox{ and for all $n=0,...,N-1$, $c^{(n+1)}$ satisfies}\\
&\dsp\int_\O \delta_\disc^{(n+\frac{1}{2})} c(\x) \Pi_\disc \varphi(\x)
+ \dsp\int_\O {\bf A}(\x,t) \nabla_\disc c^{(n+1)}(\x) \cdot \nabla_\disc \varphi(\x)\ud \x\\
&\qquad=\dsp\int_\O F(\Pi_\disc c^{(n+1)}(\x))\Pi_\disc \varphi(\x) \ud \x, \quad \forall \varphi \in X_{\disc,0}.
\end{aligned}
\end{equation}

\begin{theorem}\label{fn-th-mr}
Let assumptions \eqref{assump-rm} hold and let $(\disc_m)_{m\in\NN}$ be a sequence of gradient discretisations in the sense of Definition \ref{fn-def-gd-rm}, that is consistent, limit-conforming and compact in the sense of Definition \ref{fn-def:cons-rm}, \ref{fn-def:lconf-rm} and \ref{fn-def:compact}. For $m \in \NN$, let $c_m$ be a solution to the gradient scheme \eqref{fn-disc-pblm} with $\disc=\disc_m$. Then there exists a weak solution $\widehat c$ of \eqref{fn-weak} and a subsequence of gradient discretisations, still denoted by $(\disc_m)_{m\in\NN}$, such that, as $m \to \infty$,
\begin{subequations}
\label{eq-thm-conv}
\begin{align}
\dsp\sup_{t\in[0,T]}\| \Pi_{\disc_m}c_m(t) - \widehat c(t) \|_{L^2(\O)} \to 0,\\
|| \nabla_{\disc_m}c_m - \nabla\widehat c ||_{L^2(\O \times (0,T))^d} \to 0.
\end{align}
\end{subequations}
\end{theorem}

We can simply establish Estimates \eqref{est-rm} and \eqref{eq-est-dual} with $c$ is a solution to \eqref{fn-disc-pblm} and $\disc$ is a gradient discretisations in the sense of Definition of \ref{fn-def-gd-rm}. Therefore, the proof of the above theorem can exactly be handled as the one of Theorem \ref{th-mr}. Remark that the matches of \cite[Lemma 4.8]{30} is the following results that can be proved as \cite[Lemma 3.21]{30}.

\begin{lemma}
Let $(\disc_m)_{m\in\NN}$ be a sequence of gradient discretisations in the sense of Definition \ref{fn-def-gd-rm} which is consistent and limit-conforming in the sense of Definitions \ref{fn-def:cons-rm} and \ref{fn-def:lconf-rm}. Let  $c_m \in X_{\disc_m}^{N_m+1}$ be such that $c_m-\mathcal I_{\disc_m,\dr}g \in X_{\disc,0}^{N_m+1}$ and $(|| \nabla_{\disc_m}c_m ||_{L^2(0,T;L^2(\O)^d)})_{m\in\NN}$ is bounded. Then, there exists $c \in L^2(0,T;H^1(\O))$ (which is of trace equals $g$ on $\dr \O$), and, up to a subsequence, $\Pi_{\disc_m}c_m \to c$ weakly in $L^2(0,T;L^2(\O))$ and $\nabla_{\disc_m}u_m \to \nabla c$ weakly in $L^2(0,T;L^2(\O))^d$.  
\end{lemma}

The HMM scheme for \eqref{fn-weak} is the gradient scheme \eqref{fn-disc-pblm}
coming from the following constructed gradient discretisations.
\begin{equation*}
\begin{aligned}
X_{\disc,0}=\{ \varphi=((\varphi_{K})_{K\in \mathcal{M}}, (\varphi_{\sigma})_{\sigma \in \mathcal{E}})\;:\;{}& \varphi_{K} \in \RR,\;  \varphi_{\sigma} \in \RR,\; \varphi_{\sigma}=0 \; \mbox{for all}\; \sigma \in \edgesext\},\\
X_{\disc,\dr\O}=\{ \varphi=((\varphi_{K})_{K\in \mathcal{M}}, (\varphi_{\sigma})_{\sigma \in \mathcal{E}})\;:\;{}& \varphi_{\sigma} \in \RR,\;  \varphi_K=0 \mbox{ for all } K\in \mesh,\\
& \varphi_\sigma=0 \mbox{ for all } \sigma \in \edgesint\}.
\end{aligned}
\end{equation*}
The interpolant $\cI_{\disc,\dr}:H^{\frac{1}{2}}(\dr\O)\to X_{\disc,\dr\O}$ is defined by
\begin{equation}\label{fn-def:ID}
\begin{aligned}
\forall g\in H^{\frac{1}{2}}(\dr\O)\,:\,{}&(\cI_{\disc,\dr}g)_{\sigma}=\frac{1}{|\sigma|}\int_\sigma g(x)\ud s(x),
\mbox{ for all } \sigma \in\edgesext.
\end{aligned}
\end{equation}
The remaining discrete mappings $\Pi_\disc$, $\nabla_\disc$ and $J_{\disc}$ are as in Section \ref{sec:HMM}.

The Crouzeix--Raviart scheme for \eqref{fn-weak} is the gradient scheme \eqref{fn-disc-pblm}
coming from the following constructed gradient discretisations.
\begin{equation*}
\begin{aligned}
X_{\disc,0}=\{ \varphi=(v_{\sigma})_{\sigma \in \mathcal{E}}\;:\;{}& \varphi_{K} \in \RR,\;   \varphi_{\sigma}=0 \; \mbox{for all}\; \sigma \in \edgesext\},\\
X_{\disc,\dr\O}=\{ \varphi=(\varphi_{\sigma})_{\sigma \in \mathcal{E}}\;:\;{}& \varphi_{\sigma} \in \RR,\;  \varphi_\sigma=0 \mbox{ for all } \sigma \in \edgesint\}.
\end{aligned}
\end{equation*}
The interpolant $\cI_{\disc,\dr}$ is defined as in \eqref{fn-def:ID}. The remaining discrete mappings $\Pi_\disc$, $\nabla_\disc$ and $J_{\disc}$ are defined as in Section \ref{sr-scheme}.
It is proved in \cite[Chapter 9]{30} that the gradient discretisations of the case of non homogeneous Dirichlet boundary conditions satisfy the three properties; the consistency, the limit conformity and the compactness. Therefore, Theorem \ref{fn-th-mr} provides the convergence of the HMM and Crouzeix--Raviart schemes for the problem \eqref{fn-weak}.

\section{Numerical results}\label{sec-nr}
\par We employ the HMM method presented previously to model how a brain tumor invades an anisotropic environment, as simulated by the macroscopic reaction diffusion model \eqref{problem-rm}, where $\bar c$ represents the density of cancerous cells at a given time, $t$, and brain location, $\x=(x,y)$, which belongs to the square domain $\Omega=[-L,L]^2$, with non-flux boundary conditions. The forward Euler method is utilized for the time-discretization, incorporating the use of time step $\delta t$. In accordance with \cite{21}, we utilize the initial data, incorporating several distorted Gaussian functions that are focused on a range of points chosen as
\begin{equation}
\begin{aligned}
\bar{c}(\x,0)=0.8\exp(-0.1(x&-1)^2-0.3(y-3)^2)\\&+0.75\exp(-0.25(x-10)^2-0.15(y+9)^2)\\
&+0.6\exp(-0.2(x+3)^2-0.5(y+4)^2)\\
&+0.5\exp(-0.25(x+5)^2-0.3(y-1)^2).
\end{aligned}
\label{eq6}
\end{equation}

\subsection{Description of the underlying model}
\subsubsection{The diffusion term}
\par{}
Returning to the reaction model \eqref{problem-rm}, anisotropic diffusion ${\bf A}$ has a specific symmetric and positive definite matrix and is regarded as being in direct proportion to the measured water diffusion tensor ($DT$) harvested by DTI data \cite{19,20}. $DT$ details water molecule diffusion employing a Gaussian model. For a.e. $\x\in\O\subset\RR^2$, $DT$ can be represented by \cite{19}
\begin{equation}\label{eq2}
DT(\x)=\lambda_1\phi_1\phi_1^T+ \lambda_2\phi_2\phi_2^T,
\end{equation}
the non negative $\lambda_1$ and $\lambda_2$ are the eigenvalues of ${\bf A(\x)}$, and $\phi_1$ and $\phi_2$ are (orthogonal and normalised) eigenvectors corresponding to $\lambda_1$ and $\lambda_2$ respectively. The axis of dominating anisotropy is indicated by the eigenvector $\phi_{1}$ and the degree of anisotropy is determined by the eigenvalue size \cite{19,20}. For a.e. $\x\in\O\subset \RR^2$, computing the diffusion tensor ${\bf A(\x)}$ from \eqref{eq2} is then given as \cite{19}
\begin{equation}\label{eq3}
{\bf A}(\x)=\frac{r^2}{\mu}((\delta+(1-\delta)(1-\frac{I_2(M(\x))}{I_0(M(\x))}))I+2(1-\delta)\frac{I_2(M(\x))}{I_0(M(\x))}\phi_1\phi_1^T).
\end{equation}
For the details of computations, we refer the reader to \cite{20}. 
${\bf A}$ is divided into isotropic (the I term) and anisotropic components; the anisotropy points in the direction $\pm\phi_1$. These terms' relative sizes are dictated by $\delta$ and the function $M$ \cite{19}. The function $M$ details the concentration levels surrounding the dominant direction, for a.e. $\x\in\O\subset\RR^2$, it is represented by \cite{19}
\begin{equation}
M(\x)=\kappa F_{anis}(DT(\x)),\label{eq4}
\end{equation}
where $\kappa$ being a proportionality constant denoting how sensitive the cells are to directional data within the environment; $F_{anis}$ represents fractional anisotropy, the most commonly employed measure for anisotropy. For a.e. $\x\in\O\subset\RR^2$, the fractional anisotropy defined as \cite{19}:
\begin{equation}
F_{anis}(DT(\x))=\frac{|\lambda_{1}-\lambda_{2}
|}{\sqrt{\lambda_{1}^{2}+\lambda_{2}^{2}}}.\label{eq5}
\end{equation}

We can think of this as representing the difference between a perfect sphere and the ellipsoid form of the tensor \cite{19}. 
The function $I_j$ denotes the modified Bessel function of first kind of order $j$. The constants $\mu$ and $r$ respectively represent a tumor cell's turning rate and constant average speed \cite{19,20}.
The diffusion tensor ${\bf A}$ in our model is thus taken here to be an anisotropic informed by DTI diffusion tensor ($DT$) given by \eqref{eq3} and $DT$ is chosen as \cite{19} 
\begin{equation}
DT(\x)=\begin{pmatrix}
0.5-d(x,y)&0\\
0&0.5+d(x,y)
\end{pmatrix},
\label{eq7}
\end{equation}
where $d(x,y)=0.25\exp(-0.05x^2)-0.25\exp(-0.5y^2)$ and with parameters are chosen as $\delta=0.05$, $r=1$ and $\mu=0$.

\subsubsection{The reaction term}
With the proliferation term $F(\bar c)$ that details a population of tumor cells' growth, the majority of researchers have regarded it as a quadratic function that may be exponential: $F(\bar c)=\rho \bar{c}$ with constant proliferation rate $\rho$, illustrating that cellular division is obedient to a cycle, or logistic: $F(\bar c)=\rho \bar{c}(1-\bar{c})$ with a decrease in $\rho$, the proliferation parameter, for zones of high cellular density. Nevertheless, in this research, following \cite{21}, cell growth is modeled with a cubic behavior so that the influence of the threshold of cancer cell density can be captured, something that is ignored when a quadratic function is employed. The threshold is a local driver of normal tissue for cancer cell regimes or will not permit cancer cells to colonize tissues, keeping it free of tumors, and this is an important factor to consider. Thus we consider a reaction term of \cite{21,22}:
\[
F(\bar c)=\rho \bar{c}(1-\bar{c})(\bar{c}-\alpha),
\]
where $0<\alpha<1$ represents the cancer generation threshold. In our work, this threshold was fixed as $\alpha=0.1$ with a rate of proliferation of $\rho=1$.
This is the standard measurement used in bistable equation models, as highlighted in \cite{21}.

\subsection{Tumour growth evolution}
We subsequently use simulation of \eqref{problem-rm} to assess the impact that the anisotropic environment has on the spatio-temporal evolution of the glioma tumor at various measurements of $\kappa$. 
 The study is performed in the context of a 2D squared domain $\Omega=[-20,20]^2$ that incorporates a uniform mesh that consists of $3584$ triangular elements, that are taken at different times, denoted by  $T=0.2$, $T=5$, $T=15$ and $T=20$. The results of simulations are shown in Fig. \ref{fig-t1}, Fig. \ref{fig-t2}, Fig. \ref{fig-t3} and Fig. \ref{fig-t4} for $\kappa=0$,  $\kappa=5$, $\kappa=10$ and $\kappa=30$, respectively. In the case of the isotropic case ($\kappa=0$), the outcomes of the HMM approach are aligned with the findings outlined in \cite{21}. We expanded the research to include the time considerations for the invasion rate in an anisotropic environment, as depicted in Fig. \ref{fig-t2}, Fig. \ref{fig-t3} and Fig. \ref{fig-t4}. 
  The figures show the density of cancerous cells as a function of time indicating by color-map that is displayed on the figures, which spans from the lowest density, blue ($\bar{c}=0$), to the highest density, red ($\bar{c}=1$). 
The simulation commences from a relatively irregular inhomogeneous initial data \eqref{eq6}. As time passes, there is a tendency for some of the regions of cancer to reduce in density. However, by $T=20$, the colonization of one of the peaked populations is observed. The way in which these cancer cells are diffused in a heterogeneous environment within the brain is also taken into consideration. The isotropic case, $\kappa=0$, can be observed in  
Fig\ref{fig-t1}. In this case, the tumour cells are equally diffused in a myriad of directions and exhibit spherical symmetry. When $\kappa$ is enhanced from zero the effect that the environmental anisotropy has on cell turning increases and an increasingly non-uniform invasion is observed, as can be seen in Fig. \ref{fig-t2} for $\kappa=5$ , Fig. \ref{fig-t3} for $\kappa=10$ and Fig. \ref{fig-t4} for $\kappa=30$. For instance, at $T=20$, with $\kappa=0$ a circular invasion can be observed. This becomes increasingly irregular in shape when $\kappa=30$. The outcome of this is a distinct disparity between the anisotropic and isotropic cases. As can be clearly seen, the HMM simulations results, in terms of the expansion of the brain tumours, are consistent with the results in \cite{19}.

\subsection{Measure of diffusion anisotropy}
Researchers have recommended a range of eigenvalue-based formulae for the purposes of measuring the anisotropy in the diffusion tensor; for example, the fractional anisotropy ($F_{anis}$) measure outlined in \eqref{eq5}. $F_{anis}$ was named such due to the fact that it measures the anisotropic fraction of the diffusion. This measure can be perceived as representative of the variation between a perfect sphere and the ellipsoid form of the tensor.  
  To investigate this point in more depth, in Fig. \ref{fig-t5} we plot the histogram of $F_{anis}$ for ${\bf A}$ under (a) $\kappa=0$,  (b) $\kappa=5$, (c) $\kappa=10$ and (d) $\kappa=30$ at the time $T=10$. Other parameters are set as in Fig. \ref{fig-t1}. $F_{anis}$ is normalized to attain values in $[0,1]$ with $F_{anis}=0$ corresponding to the isotropic case as in (a) for $\kappa=0$ and $F_{anis}=1$ denoting a completely anisotropic case as clearly observed in (d) for $\kappa=30$. Moreover, the maximum value of the $F_{anis}$ increases from around of $0.6$ to approximately $0.8$ as the parameter $\kappa$ is increased from $\kappa=5$ to $\kappa=10$. This demonstrates that there is a strong dependency on the $\kappa$ anisotropy enhancement parameter.

\subsection{Anisotropy under different Meshes}
The most robust feature of the HMM approach is that it can be employed on multiple forms of meshes in a variety of space dimensions, with relatively limited limitations on the control volumes, see the numerical tests described by \cite{47}.
We applied the HMM approach to our anisotropic reaction diffusion model for four different types of meshes, presented in Fig. \ref{fig-t7} for $\kappa=0$ and in Fig. \ref{fig-t8} for $\kappa=100$ at time $T=2$. Other parameters are chosen as in Fig. \ref{fig-t1}. 
The first two types (a and b) were constructed on a triangular mesh that included $3584$ cells, and on a rectangular mesh that included 1024 control volumes. The third type was constructed on hexagonal cells, with the number of cells equal to $6561$. The fourth type of mesh was inspired by \textit{Kershowa} \cite{47} and it includes $4624$ control volumes. While this form of mesh is associated with intense distortions, the dynamics of the brain tumor is still captured by the HMM scheme, even in the event of a highly heterogeneous anisotropic case; for example, that presented in Fig. \ref{fig-t8} for $\kappa=100$.

\begin{figure}[ht]
	\begin{center}
	\includegraphics[scale=0.70]{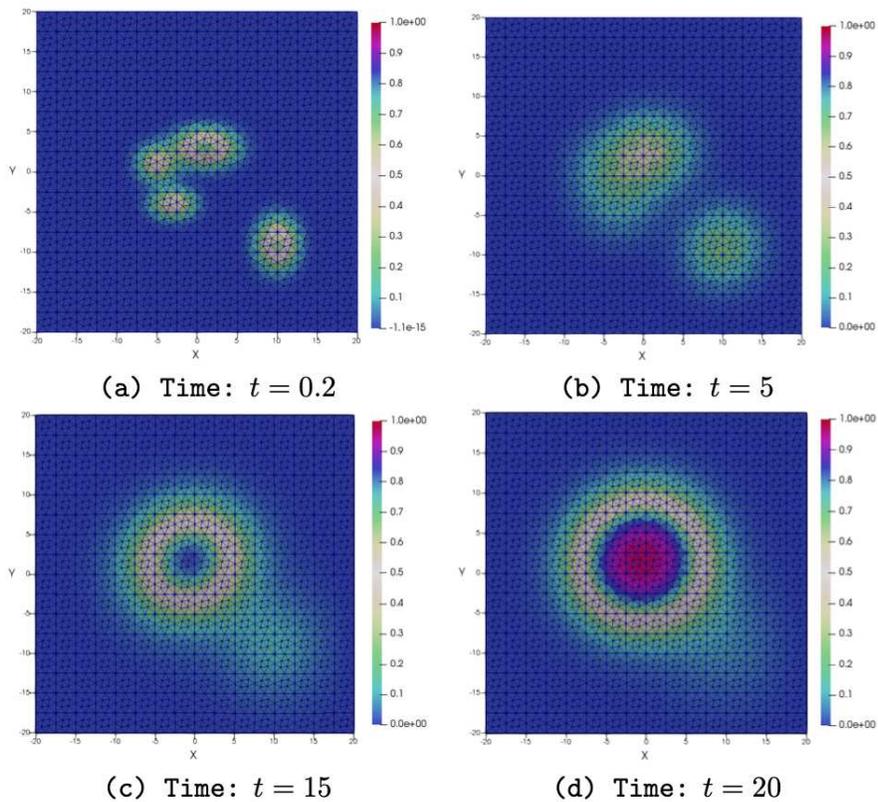}
	\end{center}
	\caption{Isotropic case with $\kappa=0$ , propagation of brain tumour at different time levels: $T=0.2$, $T=5$, $T=15$ and $T=20$.}
	\label{fig-t1}
\end{figure}

\begin{figure}[ht]
	\begin{center}
	\includegraphics[scale=0.70]{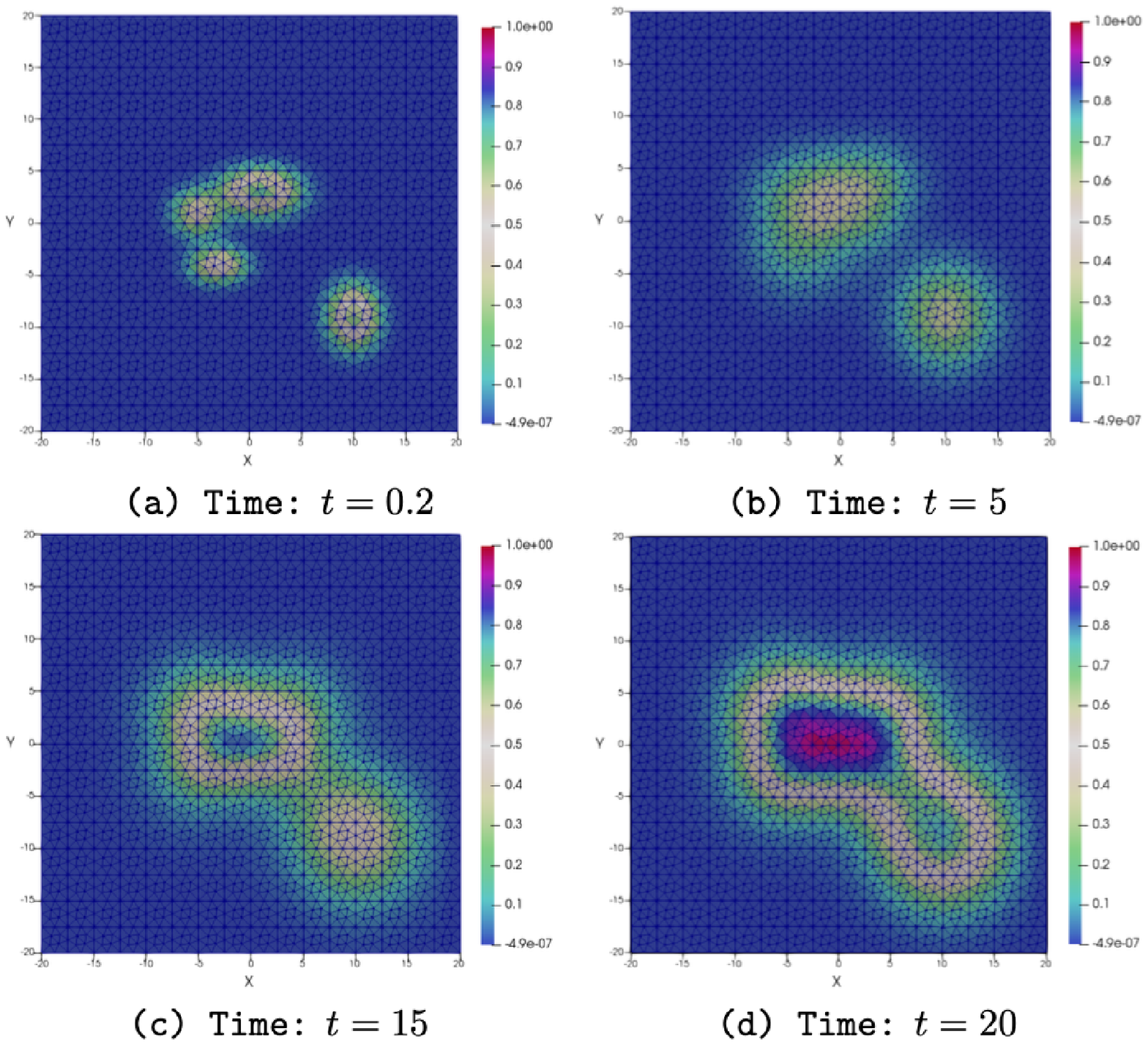}
	\end{center}
	\caption{Anisotropic case with $\kappa=5$ , propagation of brain tumour at different time levels: $T=0.2$, $T=5$, $T=15$ and $T=20$.}
	\label{fig-t2}
\end{figure}

\begin{figure}[ht]
	\begin{center}
	\includegraphics[scale=0.70]{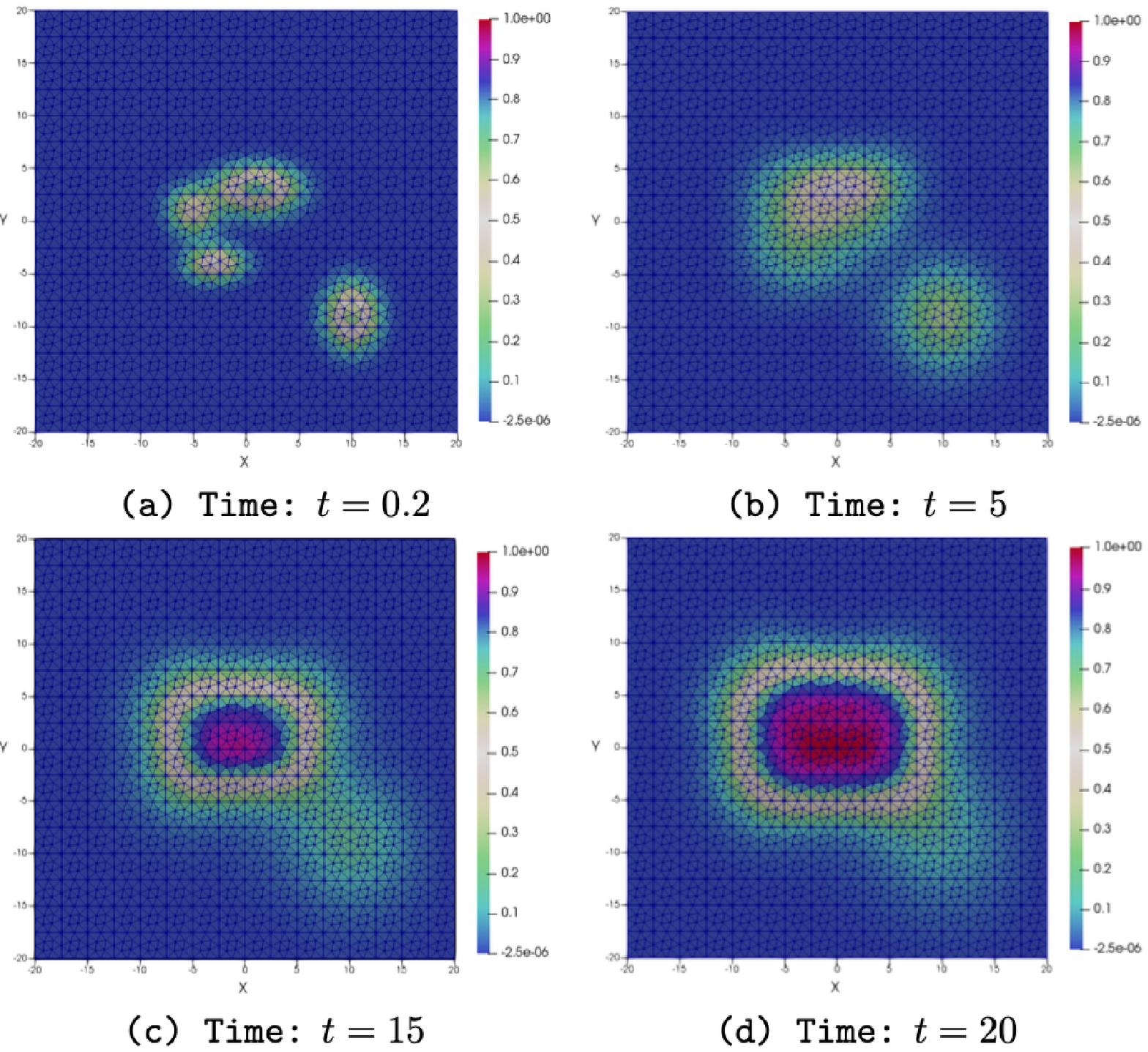}
	\end{center}
	\caption{Anisotropic case with $\kappa=10$ , propagation of brain tumour at different time levels: $T=0.2$, $T=5$, $T=15$ and $T=20$.}
	\label{fig-t3}
\end{figure}

\begin{figure}[ht]
	\begin{center}
	\includegraphics[scale=0.70]{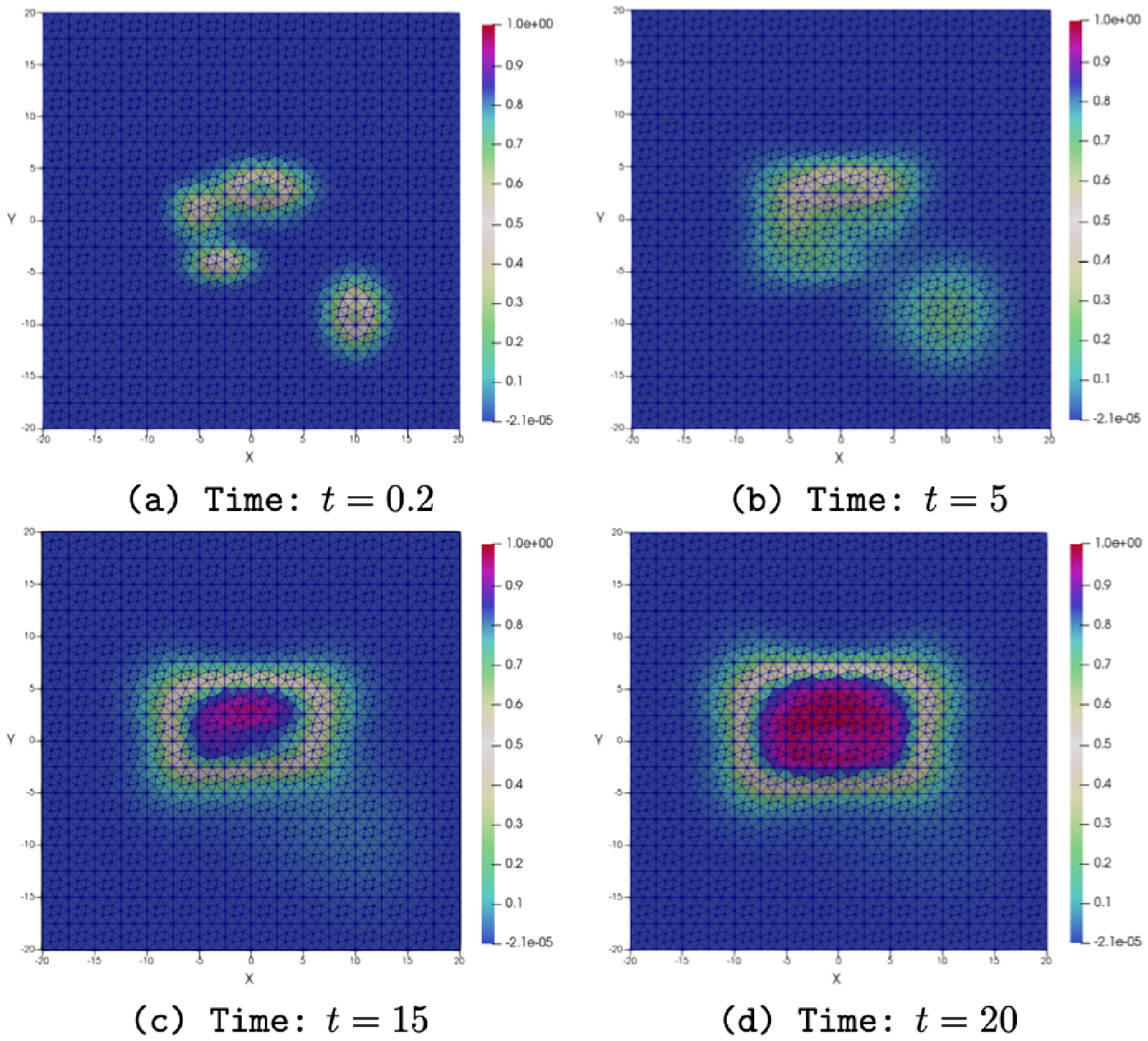}
	\end{center}
	\caption{Anisotropic case with $\kappa=30$ , propagation of brain tumour at different time levels: $T=0.2$, $T=5$, $T=15$ and $T=20$.}
	\label{fig-t4}
\end{figure}

\begin{figure}[ht]
	\begin{center}
	\includegraphics[scale=0.60]{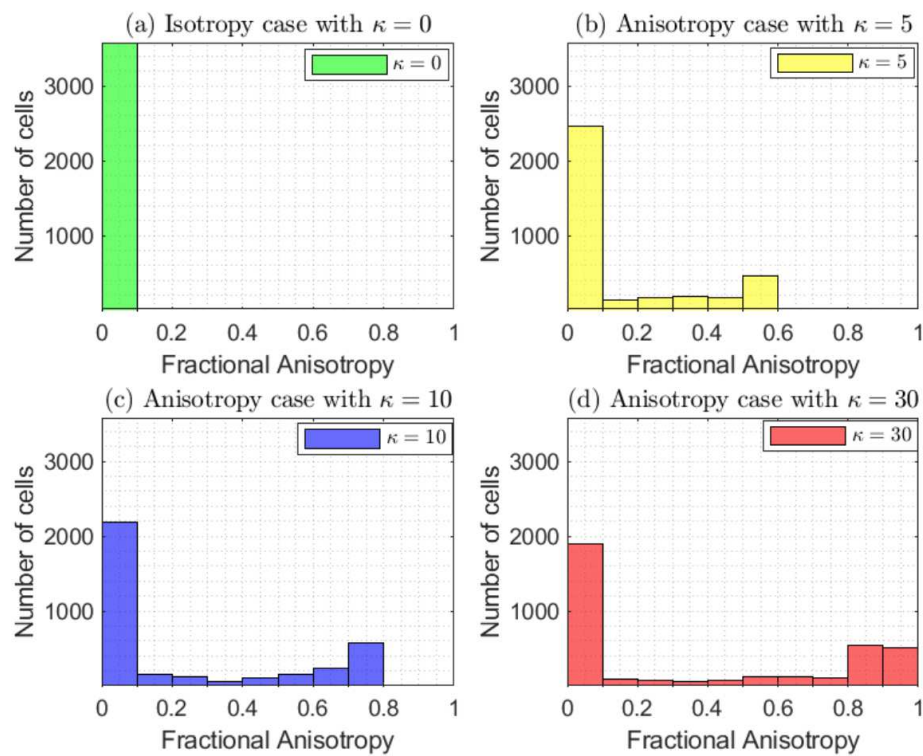}
	\end{center}
\caption{Histogram of fractional anisotropy of the domain cells under different values of $\kappa$ at time $t=10$.}
\label{fig-t5}
\end{figure}
\begin{figure}[ht]
	\begin{center}
	\includegraphics[scale=0.70]{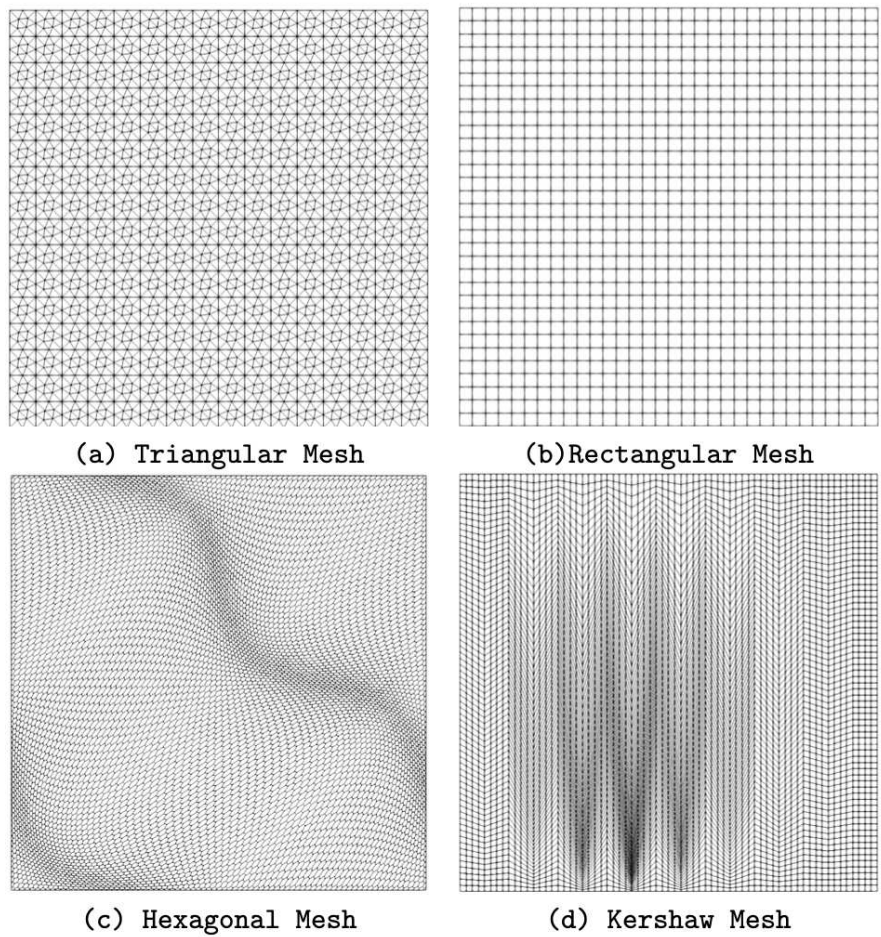}
	\end{center}
	\caption{Samples of the various 2D meshes}
	\label{fig-t6}
\end{figure}
\begin{figure}[ht]
	\begin{center}
	\includegraphics[scale=0.70]{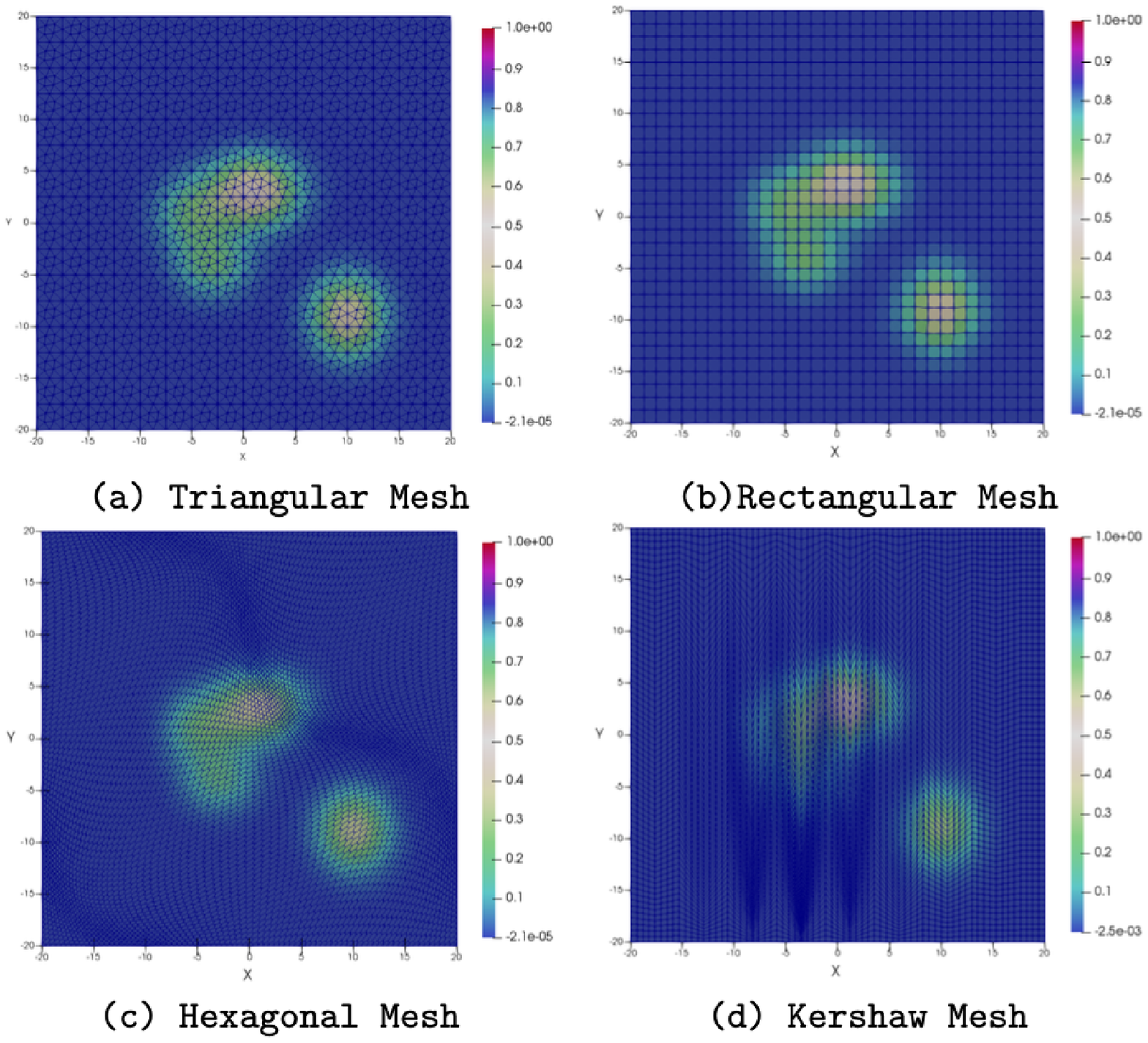}
	\end{center}
	\caption{Isotropic case with $\kappa=0$ ,  simulation of brain tumour at $t=2$ and on different types of meshes.}
	\label{fig-t7}
\end{figure}

\begin{figure}[ht]
	\begin{center}
	\includegraphics[scale=0.70]{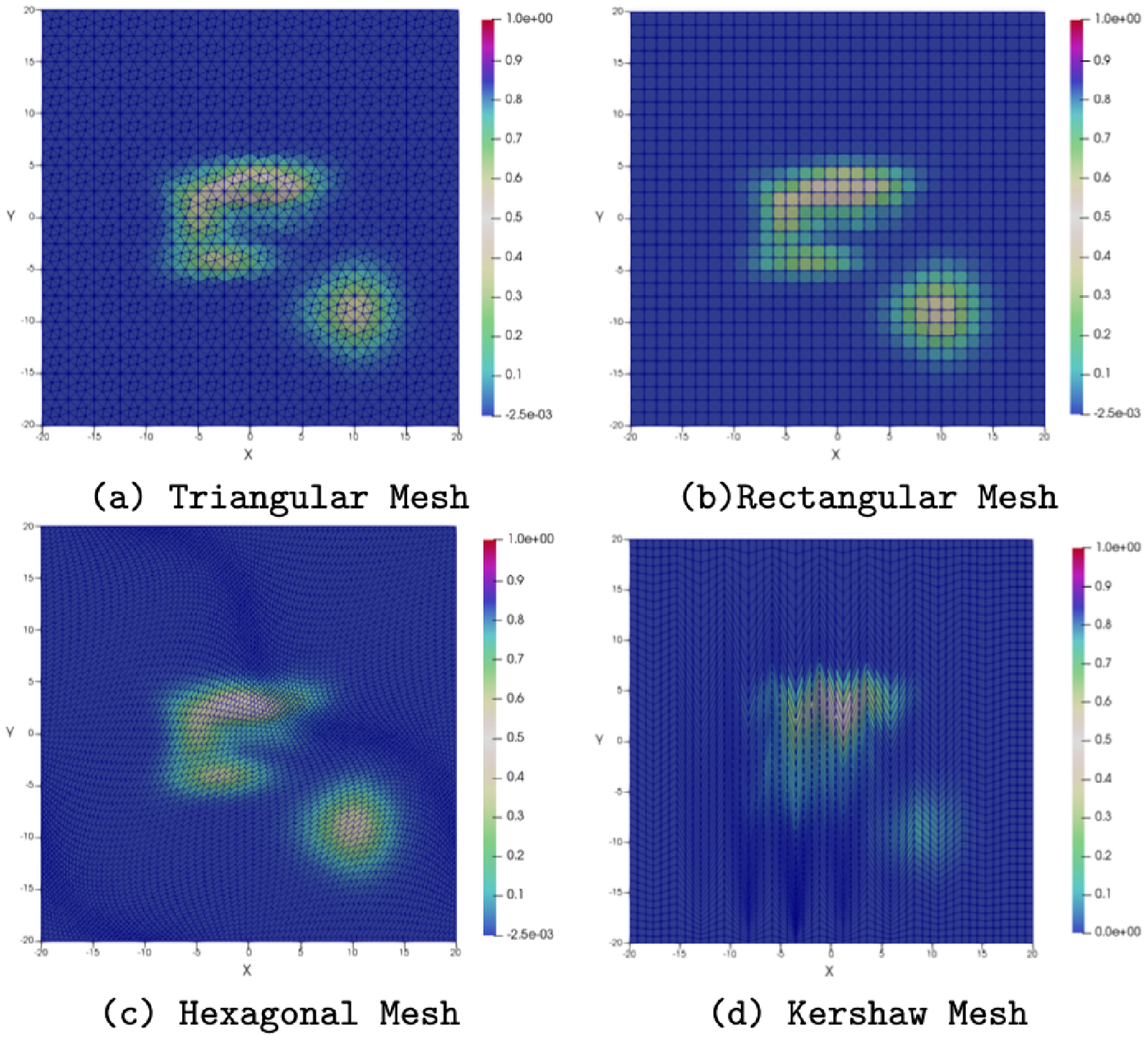}
	\end{center}
	\caption{Anisotropic case with $\kappa=100$ , simulation of brain tumour at $t=2$ and on different types of meshes.}
	\label{fig-t8}
\end{figure}         


\bibliographystyle{siam}
\bibliography{rm-ref}

\end{document}